\documentclass[12pt,a4paper]{amsart}
\usepackage{amsmath}
\usepackage{amssymb}
\usepackage{amsfonts}
\usepackage{xcolor}

\usepackage{mathrsfs}

\newtheorem{theorem}{Theorem}[section]

\newtheorem{proposition}[theorem]{Proposition}

\newtheorem*{conjecture*}{Conjecture}
\newtheorem*{problem*}{Problem}

 \theoremstyle{remark}
\newtheorem*{remark}{Remark}

\numberwithin{equation}{section}
\DeclareMathOperator{\Gal}{Gal}

\DeclareMathOperator{\Ad}{Ad}

\DeclareMathOperator{\ord}{ord}
\DeclareMathOperator{\Sym}{Sym}

\DeclareMathOperator{\GL}{GL}
\DeclareMathOperator{\SL}{SL}

\usepackage{mathtools}
\DeclarePairedDelimiter{\ceil}{\lceil}{\rceil}

\DeclareMathOperator{\lD}{\underline{\delta}}
\DeclareMathOperator{\mcS}{\mathcal{S}}

\renewcommand{\Re}{{\mathfrak{Re}}}

\begin{document}

\title{Refinements of Strong Multiplicity One for  $\mathrm{GL}(2)$}
\author{Peng-Jie Wong}
\address{Department of Mathematics and Computer Science\\
University of Lethbridge\\
4401 University Drive\\
Lethbridge, Alberta\\
T1K 3M4 Canada}
\email{pengjie.wong@uleth.ca}

\dedicatory{In memory of Professor Richard Guy}


\thanks{The author was supported by a PIMS postdoctoral fellowship and the University of Lethbridge.}
\begin{abstract}
For distinct unitary cuspidal automorphic representations $\pi_1$ and $\pi_2$ for $\GL(2)$ over a number field $F$ and any $\alpha\in\Bbb{R}$, let $\mcS_{\alpha}$ be the set of  primes $v$ of $F$ for which $\lambda_{\pi_1}(v)\neq e^{i\alpha} \lambda_{\pi_2}(v)$, where $\lambda_{\pi_i}(v)$ is the Fourier coefficient  of $\pi_i$ at $v$. In this article, we show that the lower Dirichlet density of $\mcS_\alpha$ is at least $\frac{1}{16}$.   Moreover, if $\pi_1$ and $\pi_2$ are not twist-equivalent, we show that the lower Dirichlet densities of $\mcS_\alpha$ and $ \cap_\alpha\mcS_\alpha$ are at least $\frac{2}{13}$ and $\frac{1}{11}$, respectively. Furthermore,  for non-twist-equivalent $\pi_1$ and $\pi_2$, if each $\pi_i$ corresponds to a non-CM newform of weight $k_i\ge 2$ and with trivial nebentypus, we obtain various upper bounds for the number of primes $p\le x$ such that $\lambda_{\pi_1}(p)^2 = \lambda_{\pi_2}(p)^2$. These present refinements of the works of Murty-Pujahari, Murty-Rajan, Ramakrishnan, and Walji.
\end{abstract}
\maketitle

\section{Introduction and statement of main results}

Let $F$ be a number field and $\Bbb{A}_{F}$ be its ad\`ele ring. Let $\pi_1$ and $\pi_2$ be (unitary)  cuspidal automorphic representations for $\GL_2(\Bbb{A}_F )$. For each $i$, at any unramified (finite) prime  $v$ of $F$ for  $\pi_i$, we denote the trace of  the Langlands conjugacy class of $\pi_i$  by $\lambda_{\pi_i}(v)$. There is a question of determining whether $\pi_1$ and $\pi_2$ are globally equivalent (i.e.  $\pi_1\simeq \pi_2$) from the local information on $\pi_1$ and $\pi_2$. For instance, given the set
$$
\mathcal{S}_0=\mathcal{S}_0(\pi_1,\pi_2) =\{\text{$v$ unramified for both $\pi_1$ and $\pi_2$} \mid \lambda_{\pi_1}(v)\neq  \lambda_{\pi_2}(v) \},
$$
what information on $\mathcal{S}_0$ would be sufficient for one to determine the global equivalence of $\pi_1$ and $\pi_2$? In \cite{JS81}, Jacquet and  Shalika showed that if $\mathcal{S}_0$ is finite, then $\pi_1\simeq \pi_2$ (which is often called the strong multiplicity one theorem). This was improved by  Ramakrishnan \cite{Rama94}, who showed that if $\mathcal{S}_0$ is of  density less than $\frac{1}{8}$, then $\pi_1$ and $\pi_2$ are globally equivalent. By an example given by  Serre \cite{Se77} (see also \cite[Sec. 4.4]{Wa14}), such a bound is attained by  a pair of dihedral automorphic representations.\footnote{A cuspidal automorphic representation $\pi$ for $\GL_2(\Bbb{A}_F )$ is called dihedral if it admits a non-trivial self-twist, namely, there is a non-trivial Hecke character $\chi$ of $F$ such that $\pi\otimes \chi \simeq \pi$. Gelbart and  Jacquet  \cite{GJ} showed that $\Ad(\pi)$ is cuspidal if $\pi$ is  non-dihedral.} 

Naturally, one may ask if the bound can be improved if dihedral automorphic representations are excluded. In \cite{Wa14}, Walji gave an affirmative answer by proving that if   $\pi_1$ and $\pi_2$ are distinct and non-dihedral, then  $\underline{\delta}(\mathcal{S}_0) \ge  \frac{1}{4}$,
where  $\underline{\delta}(\mathcal{S}_0)$ denotes the lower Dirichlet density of $\mathcal{S}_0$.\footnote{We recall that the lower Dirichlet density $\underline{\delta}(A)$ of a subset $A$ of the primes $v$ of $F$ is defined by
$$
\underline{\delta}(A)=\liminf_{s\rightarrow 1^+} \frac{\sum_{v\in A}\frac{1}{Nv^s}}{\log(\frac{1}{s-1})},
$$
where $Nv$ is the norm of $v$.}

The main result of this article is the following generalisation of the work of  Ramakrishnan \cite{Rama94} and Walji \cite{Wa14}.

\begin{theorem}\label{Main-dist}
Let  $\pi_1$ and $\pi_2$ be distinct cuspidal automorphic representations for $\GL_2(\Bbb{A}_F )$. Given $\alpha\in \Bbb{R}$, let
$$
\mathcal{S}_\alpha= \mathcal{S}_\alpha(\pi_1,\pi_2) =\{\text{$v$ unramified for both $\pi_1$ and $\pi_2$} \mid \lambda_{\pi_1}(v)\neq e^{i\alpha} \lambda_{\pi_2}(v) \},
$$
and let $\underline{\delta}(\mathcal{S}_\alpha)$ denote the lower Dirichlet density of $\mathcal{S}_\alpha$. Then 
$$
\lD(\mathcal{S}_\alpha)
\ge
\begin{cases}
\frac{1}{6+ 2\cos(2\alpha) } & \text{if $\cos(2\alpha)\ge 0$ and $\cos\alpha\ge 0$;}\\
\frac{1}{6+ 2\cos(2\alpha) -8\cos\alpha}  & \text{if  $\cos(2\alpha)\ge 0$ and $\cos\alpha \le 0$;}\\
\frac{1}{6}  & \text{if  $\cos(2\alpha)\le 0$ and $\cos\alpha \ge 0$;}\\
\frac{1}{6 -8\cos\alpha}  & \text{if  $\cos(2\alpha)\le 0$ and $\cos\alpha \le 0$.}
\end{cases}
$$

Moreover, if both $\pi_1$ and $\pi_2$ are non-dihedral, then 
$$
\lD(\mathcal{S}_\alpha)
\ge  \begin{cases} 
 \min\big\{ \frac{1}{3+ \cos (2\alpha)} ,\frac{1}{ 3 +\cos (2\alpha)  - 2 \kappa_1 \cos \alpha} \big \}  & \text{if $\cos(2\alpha)\ge 0$;}\\
 \min\big\{ \frac{1}{3} ,\frac{1}{ 3   - 2 \kappa_1 \cos \alpha} \big \}   & \text{if $\cos(2\alpha)\le 0$,}
 \end{cases}
$$
where  $\kappa_1$ is 1  if $\pi_1\simeq \pi_2 \otimes \nu$ for some cubic Hecke character $\nu$ and 0 otherwise.
\end{theorem}

\begin{remark}
Let $n\ge 3$, and let $\pi_1$ and $\pi_2$ be distinct cuspidal automorphic representations for $\GL_n(\Bbb{A}_F )$, satisfying the Ramanujan-Petersson conjecture. It can be shown that $
\lD(\mcS_\alpha)\ge  \frac{1}{2n^2}.
$
In addition, if both $\Ad(\pi_1)$ and $\Ad(\pi_2)$ are cuspidal, then  $\lD(\mcS_\alpha)\ge \frac{1}{8}$. (See the final remark in Section \ref{proof-lD_S*} for more details.) Also, as pointed out by the referee,  Walji \cite{Wa20} obtained $\lD(\mcS_0)\ge \frac{1}{8}$ and, assuming further that $\Ad(\pi_1)$ and $\Ad(\pi_2)$ are distinct, $\lD(\mcS_0)\ge \frac{1}{3+2\sqrt{2}}$.
\end{remark}

We shall note that our interest in this theorem was motivated by the following theorem of Ramakrishnan \cite[Corollary]{Rama00} and \cite[Corollary 4.1.3]{Rama-ann-00}, which uses the information on $\mathcal{S}_0 \cap \mathcal{S}_{\pi}$ to determine whether two given $\GL(2)$-forms are twist-equivalent.

\begin{theorem}[Ramakrishnan]\label{Rama}
Let  $\pi_1$ and $\pi_2$ be cuspidal automorphic representations for $\GL_2(\Bbb{A}_F )$. If $\pi_1$ and $\pi_2$ are with trivial central character, and
$$
\lambda_{\pi_1}(v)^2 = \lambda_{ \pi_2}(v)^2
$$
for all unramified primes (for both $\pi_1$ and $\pi_2$), then $\pi_1$ and $\pi_2$ are twist-equivalent. 

Moreover, if $\pi_1$ and $\pi_2$ correspond to holomorphic  newforms (over $\Bbb{Q}$) of same weight and with same nebentypus, and 
$$
\lambda_{\pi_1}(p)^2 = \lambda_{ \pi_2}(p)^2
$$
outside a set $S$ of primes $p$ of density less than $\frac{1}{18}$, then $\pi_1$ and $\pi_2$ are twist-equivalent. Furthermore, if both $\pi_1$ and $\pi_2$ are non-dihedral, then one has the same result assuming only that the density of $S$ is less than 1. 
\end{theorem}

In light of this theorem, one may also ask a question of determining whether $\pi_1$ and $\pi_2$ are twist-equivalent from the information on a single $\mathcal{S}_\alpha$. For $\alpha=0$,  Walji \cite{Wa14} proved the following theorem.

\begin{theorem}[Walji]\label{Wa-thm}
Let $\pi_1$ and $\pi_2$ be non-twist-equivalent cuspidal automorphic representations for $\GL_2(\Bbb{A}_F )$ with unitary central characters. Then
\begin{enumerate}
 \item[(i)] if both $\pi_1$ and $\pi_2$ are dihedral, then $\underline{\delta}(\mathcal{S}_0) \ge \frac{2}{9}$;
 \item[(ii)] if exactly one of $\pi_1$ and $\pi_2$  is non-dihedral, then $\underline{\delta}(\mathcal{S}_0) \ge \frac{2}{7}$;
\item[(iii)] if both $\pi_1$ and $\pi_2$ are non-dihedral, then $\underline{\delta}(\mathcal{S}_0) \ge \frac{2}{5}$.
\end{enumerate}
\end{theorem} 

The second objective of this article is to prove the following ``rotation variant'' of the work of Walji, Theorem \ref{Wa-thm}.

\begin{theorem}\label{main-thm-non-twist-equiv}
Let $\alpha\in \Bbb{R}$, and let $\pi_1$ and $\pi_2$ be non-twist-equivalent cuspidal automorphic representations for $\GL_2(\Bbb{A}_F )$ with unitary central characters $\omega_1$ and $\omega_2$, respectively. Then
\begin{enumerate}
 \item[(i)] if both $\pi_1$ and $\pi_2$ are dihedral, then $\lD(\mcS_\alpha) \ge \min \{d_\alpha, \frac{1}{4}\}$, where
$$
d_\alpha =
\begin{cases}
 \frac{2}{7 + 2 \cos(2\alpha) }  & \text{if $\cos(2\alpha)\ge 0$ and $\cos\alpha\ge 0$;}\\
\frac{2}{7 + 2 \cos(2\alpha)  -4 \cos \alpha}  & \text{if  $\cos(2\alpha)\ge 0$ and $\cos\alpha \le 0$;}\\
  \frac{2}{7}  & \text{if  $\cos(2\alpha)\le 0$ and $\cos\alpha \ge 0$;}\\
\frac{2}{7  -4 \cos \alpha}   & \text{if  $\cos(2\alpha)\le 0$ and $\cos\alpha \le 0$;}
\end{cases}
$$

\item[(ii)] if exactly one of $\pi_1$ and $\pi_2$  is non-dihedral, then
$$
\lD(\mcS_\alpha) \ge
\begin{cases}
\frac{2}{ 5 +2   \cos (2\alpha)  }  & \text{if $\cos(2\alpha)\ge 0$;}\\
\frac{2}{5} & \text{if $\cos(2\alpha)\le 0$;}
\end{cases}
$$ 
\item[(iii)] if both $\pi_1$ and $\pi_2$ are non-dihedral, then $\underline{\delta}(\mathcal{S}_\alpha) \ge  \frac{2}{ 4+ \kappa_2 \cos (2\alpha)} $, where  $\kappa_2$ is 1  if $\omega_1= \omega_2$ and 0 otherwise.
\end{enumerate}
\end{theorem} 

Furthermore, we have the following refinement of Theorem \ref{Rama}.

\begin{theorem}\label{lD_S*}
 Let $\pi_1$ and $\pi_2$ be non-twist-equivalent cuspidal automorphic representations for $\GL_2(\Bbb{A}_F )$ with unitary central characters, and let
$$
\mcS_*= \mcS_* (\pi_1,\pi_2)= \{\text{$v$ unramified for both $\pi_1$ and $\pi_2$} \mid |\lambda_{\pi_1}(v)|\neq  |\lambda_{\pi_2}(v)| \}.
$$
Then
$$
\lD(\mcS_*)\ge
 \begin{cases} 
   \frac{1}{8}  & \text{if  both $\pi_1$ and $\pi_2$ are dihedral;}\\
   \frac{1}{9.58}  & \text{if exactly  one of $\pi_1$ and $\pi_2$ is dihedral;}\\ 
    \frac{1}{10.76} & \text{if  both $\pi_1$ and $\pi_2$ are non-dihedral.}
 \end{cases}
$$

\end{theorem} 

Consequently, we have the following refined version of \cite[Theorem 4.1.2]{Rama-ann-00} asserting that if the adjoint lifts of $\pi_1$ and $\pi_2$ agree at almost all primes, then $\pi_1$ and $\pi_2$ are twist-equivalent.

\begin{theorem}\label{Ad-thm}
 Let $\pi_1$ and $\pi_2$ be  cuspidal automorphic representations for $\GL_2(\Bbb{A}_F )$ with unitary central characters. If 
 $$
 \Ad(\pi_{1,v})\simeq   \Ad(\pi_{2,v})
 $$
outside a set of lower Dirichlet density less than $ \frac{1}{10.76}$, then 
$$
\pi_2 \simeq \pi_1 \otimes \chi
$$ 
for some id\`ele class character. 
\end{theorem}

\begin{proof}[Proof of Theorem  \ref{Ad-thm}]
Toward a contradiction, suppose that $\pi_1$ and $\pi_2$ are not twist-equivalent. By Theorem  \ref{lD_S*}, we know that $\lD(\mcS_*)\ge \frac{1}{10.76}$. However, from our assumption of the theorem, it follows that  $
\lambda_{\Ad( \pi_1)}(v) = \lambda_{\Ad( \pi_2)}(v) 
$
outside a set of lower Dirichlet density less than $ \frac{1}{10.76}$. In other words, the  lower Dirichlet density of the set of primes $v$ for which
$
\lambda_{\Ad( \pi_1)}(v) \neq \lambda_{\Ad( \pi_2)}(v) 
$
is less than $ \frac{1}{10.76}$. This, together with the fact that $ |\lambda_{\pi_i}(v)|^2 = \lambda_{\Ad( \pi_i)}(v) +1 $ (for unramified $v$), leads to 
$$
\frac{1}{10.76} \le  \lD(\mcS_*) =   \lD(\{v \mid \lambda_{\Ad( \pi_1)}(v) \neq \lambda_{\Ad( \pi_2)}(v)  \}) < \frac{1}{10.76},
$$
a contradiction.
\end{proof}

We also have the following interesting variant.

\begin{theorem}\label{S''} Let $\pi_1$ and $\pi_2$ be non-twist-equivalent cuspidal automorphic representations for $\GL_2(\Bbb{A}_F )$ with unitary central characters, and let
$$
\mcS^*= \mcS^* (\pi_1,\pi_2)= \{\text{$v$ unramified for both $\pi_1$ and $\pi_2$} \mid |\lambda_{\pi_1}(v)|^2 +  |\lambda_{\pi_2}(v)|^2 \neq 2 \}.
$$
Then
$$
\lD(\mcS^*)\ge
 \begin{cases} 
   \frac{1}{18}  & \text{if  $\pi_1$ and $\pi_2$ are simultaneously dihedral or non-dihedral;}\\
   \frac{1}{12}  & \text{if exactly  one of $\pi_1$ and $\pi_2$ is dihedral.}
 \end{cases}
$$   

\end{theorem} 

\begin{remark} 
(i)  As remarked in  \cite{Rama-ann-00}, if $\pi_1$ and $\pi_2$ correspond to Hilbert newforms, one  can establish a version of Theorem \ref{Rama}  by invoking  the  $\ell$-adic representations associated to $\pi_1$ and $\pi_2$ as done in \cite{Rama00}. 

(ii)  In the case that   $\pi_1 \simeq  \pi_2\otimes \chi$ for some id\`ele class character $\chi$,  as $\pi_1 \boxtimes \bar{\pi}_1 \simeq \pi_2 \boxtimes \bar{\pi}_2$, $|\lambda_{\pi_1}(v)|^2 = |\lambda_{\pi_2}(v)|^2$ for any unramified $v$, and thus $\mcS_*$ is empty. Hence, to have positive $\lD(\mcS_*)$ in Theorem  \ref{lD_S*}, the non-twist-equivalence condition is necessary.

\end{remark} 

\begin{remark}  Our method is an adaption of the work of Walji \cite{Wa14}, which subtly relies on the $L$-functions associated to $\pi_1$ and $\pi_2$. The crucial ingredients are the automorphy of the adjoint lift from $\GL(2)$ to $\GL(3)$ (due to  Gelbart and Jacquet \cite{GJ}) and the  functoriality   of the tensor product  $\GL(2)\times\GL(2)\rightarrow \GL(4)$ (due to Ramakrishnan \cite{Rama-ann-00}). To prove Theorem \ref{lD_S*}, we will further require the automorphy of $\Sym^4\pi$ and its cuspidality criterion established by Kim and  Shahidi  \cite{KS00,KS02}.
\end{remark}

There are other variants and refinements of the above-mentioned work of Ramakrishnan \cite{Rama00, Rama-ann-00}. For instance, when $\pi_i$ is a cuspidal automorphic representation  corresponding to a non-CM  newform  in $S_{k_i}^{\mathrm{new}}(\Gamma_0(q_i))$  with trivial nebentypus for each $i$, by Galois-theoretic techniques, Rajan \cite[Corollary 1]{R98} showed that if 
$$
\limsup_{x\rightarrow\infty}\frac{\#\{p \le x \mid \lambda_{\pi_1}(p)p^{\frac{k_1-1}{2}} = \lambda_{\pi_2}(p)p^{\frac{k_2-1}{2}} \}}{\pi(x)}>0,
$$
then $\pi_1$ and $\pi_2$ are twist-equivalent. Also, in \cite{MP17},  Murty and  Pujahari showed that if 
$$
\limsup_{x\rightarrow\infty}\frac{\#\{p \le x \mid \lambda_{\pi_1}(p) = \lambda_{\pi_2}(p) \}}{\pi(x)}>0,
$$
then $\pi_1$ and $\pi_2$ are twist-equivalent. In much the same spirit, we have the following theorem.

\begin{theorem}\label{main-uncond} 
Let $F$ be a totally real number field.  For each $i$, let $\pi_i$ be a cuspidal automorphic representation  corresponding to a non-CM Hilbert newform   of weights $k_{i,j}\ge 2$ (at all infinite primes $v_j$ of $F$) and with trivial nebentypus. If
$$
\limsup_{x\rightarrow\infty}\frac{\#\{v  \mid Nv\le x,\enspace \lambda_{\pi_1}(v)^2 = \lambda_{\pi_2}(v)^2 \}}{\pi_F(x)}>0,
$$
where $\pi_F(x)$ denotes the number of primes $v$ of $F$ such that $Nv\le x$, then $\pi_1$ and $\pi_2$ are twist-equivalent.
\end{theorem}

It is also worth mentioning that when $\pi_1$ and $\pi_2$ correspond to non-CM newforms or  non-dihedral Maa{\ss} forms, assuming the generalised Riemann hypothesis and certain analytic properties for Rankin-Selberg $L$-functions of $\Sym^{m_1}\pi_1$ and $\Sym^{m_2}\pi_2$,  Murty and  Rajan \cite{MR96} showed that if $\pi_1$ and $\pi_2$ are not twist-equivalent, then 
\begin{equation}\label{MR}
\#\{p \le x \mid \lambda_{\pi_1}(p)= \lambda_{\pi_2}(p) \}  \ll \frac{x^{5/6}(\log( N x))^{1/3}}{(\log x)^{2/3}}
\end{equation}
for some  suitable constant $N$ (depending on $\pi_1$ and $\pi_2$). As a consequence, if the number of primes $p \le x$ for which $\lambda_{\pi_1}(p)= \lambda_{\pi_2}(p)$ is $\gg x^{\theta}$ for some $\theta>5/6$, then $\pi_1$ and $\pi_2$ are  twist-equivalent.  In light of this and Theorems \ref{Rama} and \ref{main-uncond}, it shall be reasonable to  expect that  $\pi_1  \simeq \pi_2 \otimes \chi$ for some Dirichlet character $\chi$ whenever $\lambda_{\pi_1}(p)^2= \lambda_{f_2}(p)^2$  for certain ``thin'' sets of primes $p$ (i.e., sets of \emph{zero} upper density among all primes).  We shall show that such an expectation holds, unconditionally, in certain instances as follows.

\begin{theorem}\label{main-TE-2}
For each $i$, let $\pi_i$ be a cuspidal automorphic representation  corresponding to a non-CM  newform  in $S_{k_i}^{\mathrm{new}}(\Gamma_0(q_i))$  with trivial nebentypus. If 
$$
\limsup_{x\rightarrow\infty}\frac{\#\{p \le x \mid \lambda_{\pi_1}(p)^2 = \lambda_{\pi_2}(p)^2 \}}{\pi(x)(\log\log\log x)^{1+\epsilon}/ (\log\log x)^{1/2}}>0
$$
for some $\epsilon>0$, then $\pi_1  \simeq \pi_2 \otimes \chi$ for some Dirichlet character $\chi$. 
\end{theorem}

Indeed, Theorem \ref{main-TE-2}   follows from the following  estimate.

\begin{theorem}\label{main-2} 
In the notation of Theorems \ref{main-TE-2}, if there is no Dirichlet character $\chi$ such that $\pi_1 \simeq  \pi_2 \otimes \chi$, then
$$
\#\{p \le x \mid  \lambda_{\pi_1}(p)^2= \lambda_{\pi_2}(p)^2 \} \ll
\pi(x) \frac{\log(k_1q_1k_2q_2 \log\log x)}{ \sqrt{\log\log x}}.
$$
\end{theorem}

\begin{proof}[Proof of Theorem \ref{main-TE-2}] Suppose, on the contrary, that $\pi_1$ and  $\pi_2$ are not twist-equivalent (i.e., there is no Dirichlet character such that $\pi_1 \simeq \pi_2 \otimes \chi$).  Theorem \ref{main-2} implies that
 $$
\frac{\#\{p \le x \mid \lambda_{\pi_1}(p)^2= \lambda_{\pi_2}(p)^2 \}}{\pi(x)(\log\log\log x)^{1+\epsilon}/ (\log\log x)^{1/2}}\ll  \frac{1}{(\log \log\log x)^{\epsilon}}
$$
for any $\epsilon>0$.
Taking $\limsup_{x\rightarrow\infty}$ on both sides, we see that the assumption leads to $0<0$, a contradiction. 
\end{proof}

Furthermore, assuming  the generalised Riemann hypothesis, we have the following refined version of the above-mentioned works of Murty-Rajan \eqref{MR} and  Ramakrishnan \cite{Rama00,Rama-ann-00}.
 
\begin{theorem}\label{main-GRH} In the notation of Theorem \ref{main-TE-2}, assume that  $\pi_1$ and $\pi_2$ are not twist-equivalent. If  for $(m_1,m_2)\in\{(n,n-2),(n-2,n),(n-2,n-2)\mid n\ge 3\}$, $L(s,\Sym^{m_1}\pi_1 \times \Sym^{m_2}\pi_2)$ satisfies the generalised Riemann hypothesis,   then
$$
\#\{p \le x \mid \lambda_{\pi_1}(p)^2= \lambda_{\pi_2}(p)^2 \}  \ll \frac{x^{5/6}(\log( k_1q_1k_2q_2  x))^{1/3}}{(\log x)^{2/3}}.
$$
\end{theorem}

\begin{remark} 
(i) We shall note that in order to prove Theorems \ref{main-2} and \ref{main-GRH}, we invoke the recent work of Newton and Thorne \cite{NT20} (who proved that all the symmetric powers of cuspidal automorphic representations  corresponding to non-CM  newforms  are \emph{automorphic}).

\noindent (ii) With a little more effort, it is possible to extend Theorems \ref{main-2} and \ref{main-GRH}  to  non-CM  Hilbert newforms  under the assumption of the automorphy of the pertinent symmetric powers (and the generalised Riemann hypothesis for the extension of  Theorem \ref{main-GRH}) by adapting the methods developed in \cite{Th20,PJW19}. However, for the sake of conceptual clarity, we shall not do this here. Nonetheless, in Section \ref{main-proof}, we will develop our argument in the setting of Hilbert newforms so that the corresponding results can be immediately derived once an  effective version of Proposition \ref{uncond-ST} is established.

\noindent (iii) Similar to  \cite{MP17,R98} and Theorem \ref{Rama}, it is possible to prove a version of Theorems \ref{main-uncond}, \ref{main-2}, and \ref{main-GRH} by only assuming $\pi_1$ corresponds to a newform without CM. We will discuss this in more detail in Section \ref{con-rmk}.
\end{remark}

\begin{remark} 
(i) Our method makes use of the Selberg polynomials (see Section \ref{jSTd-Sp}), which is  more elementary than the one used by Murty and Rajan \cite{MR96} (who invoked the Erd\H{o}s-Tur\'an inequality to bound the ``discrepancy'' of certain sequences, associated to $\lambda_{\pi_1}(p)$ and $\lambda_{\pi_2}(p)$, in terms
of exponential sums). Nonetheless, similar to the argument of \cite{MR96}, we apply the effective versions of the joint Sato-Tate distribution established by  Thorner \cite{Th20} and the author \cite{PJW19}.

\noindent (ii) Compared to the argument used in \cite{MP17}, ours is slightly more complicated. Nonetheless, our argument yields  better estimates. For instance, under the generalised Riemann hypothesis, using the argument of \cite{MP17} would result in
$$
\#\{p \le x \mid \lambda_{\pi_1}(p)^2= \lambda_{\pi_2}(p)^2 \} \ll_{\pi_1,\pi_2} x^{7/8}/ (\log x)^{1/2},
$$
which is not as good as the one given in Theorem  \ref{main-GRH}.
\end{remark}

The rest of article is structured as follows. In the next section, we will discuss Walji's argument as well as collecting the relevant facts that will be used later.
We will prove  Theorems \ref{Main-dist}  and \ref{main-thm-non-twist-equiv} in Sections \ref{Dih} (for pairs of dihedral representations) and  \ref{one-non-dih} (for the remaining cases); we will prove Theorems \ref{lD_S*} and \ref{S''} in Section \ref{proof-lD_S*}. The proofs of Theorems \ref{main-uncond},  \ref{main-2}, and \ref{main-GRH} will be given in Section \ref{main-proof}. In Section \ref{con-rmk}, we will discuss how to extend Theorem \ref{main-uncond} to the case that  one of the newforms is with CM.

\section{Preliminaries}\label{pre}

\subsection{$L$-functions}\label{L-functions}
We shall begin by reviewing  automorphic $L$-functions, Rankin-Selberg $L$-functions, and their properties.

Let $F$ be a number field, and let $\pi$ be a  cuspidal automorphic representation for $\GL_n(\Bbb{A}_F )$ with unitary central character, where, as later, $\Bbb{A}_F$ denotes the ad\`ele ring of $F$. We define the (incomplete) automorphic $L$-function $L(s,\pi)$ attached to $\pi$ by
$$
L(s,\pi)=\prod_v \det(I_n - A_v(\pi)Nv^{-s})^{-1},
$$
for $\Re(s)>1$, where the product is over  unramified (finite) primes $v$ for $\pi$, and $A_v(\pi)$  denotes the Langlands conjugacy class of $\pi$ at $v$. By the work of  Jacquet and  Shalika \cite{JS76}, it is known that $L(s,\pi)$ is non-vanishing at $s=1$ and with a possible simple pole at $s=1$ that only appears if $\pi$ is equivalent to the trivial id\`ele class character 1 of $F$. 

Let $\pi$ and $\pi'$ be  cuspidal automorphic representations, with unitary central characters, for $\GL_n(\Bbb{A}_F )$ and $\GL_m(\Bbb{A}_F )$, respectively. We define the (incomplete) Rankin-Selberg $L$-function $L(s,\pi\times \pi')$ attached to $\pi$ and $\pi'$ by
$$
L(s,\pi\times \pi')=\prod_v \det(I_{nm} - (A_v(\pi)\otimes A_v(\pi'))  Nv^{-s})^{-1},
$$
for $\Re(s)>1$, where the product is over unramified primes $v$ for both $\pi$ and $\pi'$. It can be shown  that $L(s,\pi\times \pi')$ extends holomorphically to $\Re(s)=1$ except for a possible simple pole at $s=1-it$ that exists only if $\pi' \simeq \bar{\pi}\otimes |\cdot|^{it}$, where $ \bar{\pi}$ is the dual of $\pi$. Moreover,  Shahidi \cite{Sh81} showed that $L(s,\pi\times \pi')$  is non-vanishing on $\Re(s)\ge 1$.

\subsection{$\GL(2)$-forms}
We also collect some fact regarding cuspidal automorphic representations for $\GL_2(\Bbb{A}_F )$ that we will make use of repeatedly throughout our discussion. 

For any cuspidal automorphic representation $\pi$ for $\GL_2(\Bbb{A}_F )$ with unitary central character $\omega$, by the work of Gelbart-Jacquet \cite{GJ}, the (automorphic) tensor product $\pi \boxtimes \bar{\pi}$ exists  as an automorphic representation for $\GL_4(\Bbb{A}_F )$ and satisfies
$$
\pi \boxtimes \bar{\pi} \simeq 1 \boxplus \Ad(\pi),
$$
where $\boxplus$ denotes the (Langlands) isobaric sum, and $\Ad(\pi)$ is an automorphic representation  for $\GL_3(\Bbb{A}_F )$. If $\pi$ is non-dihedral, then $\Ad(\pi)$ is cuspidal. Also, one knows that $\Ad(\pi)$ satisfies the relation
$$
\Ad(\pi)\simeq \Sym^2\pi \otimes \omega^{-1},
$$
where $\Sym^2\pi$ is the symmetric square of $\pi$. (We shall discuss more properties of $\Ad(\pi)$ for  dihedral representations $\pi$ in Section \ref{Dih}.)

Now, let  $\pi_1$ and $\pi_2$ be cuspidal automorphic representations for $\GL_2(\Bbb{A}_F )$ with unitary central characters.  Ramakrishnan \cite{Rama-ann-00} proved that the tensor product $\pi_1 \boxtimes \pi_2$ exists  as an automorphic representation for $\GL_4(\Bbb{A}_F )$; when both $\pi_1$ and $\pi_2$ are dihedral, one has the following cuspidality criterion of Ramakrishnan \cite{Rama-ann-00}:

\emph{$\pi_1 \boxtimes \pi_2 $ is cuspidal if and only if $\pi_1$ and $\pi_2$ cannot be induced from the same quadratic extension $K$ of $F$.}

Lastly, we recall that  from the work of   Gelbart and Jacquet   \cite{GJ}, it follows that if $\pi_1$ and $\pi_2$ are twist-equivalent (i.e., $\pi_1 \simeq \pi_2 \otimes \chi$ for some id\`ele class character $\chi$), then 
$$
1 \boxplus \Ad(\pi_1) \simeq \pi_1 \boxtimes \bar{\pi}_1 \simeq  \pi_2 \boxtimes \bar{\pi}_2 \simeq  1 \boxplus \Ad(\pi_2), 
$$
as $\chi\bar{\chi}\simeq 1$, which means that  $\Ad(\pi_1) \simeq   \Ad(\pi_2)$. Moreover, by  \cite[Theorem 4.1.2]{Rama00}, the conserve is also true. Thus, we have the following  twist-equivalence criterion:

\emph{$\pi_1$ and $\pi_2$ are twist-equivalent if and only if $\Ad(\pi_1)\simeq\Ad(\pi_2)$.}

\subsection{Walji's strategy  \cite{Wa14}}\label{strategy-Wa}

Let $F$ be a number field, and let $\pi_1$ and $\pi_2$ be  cuspidal automorphic representations
for $\GL_2(\Bbb{A}_F )$ with (unitary) central characters $\omega_1$ and $\omega_2$, respectively. As before, for each $i$, we denote the trace of  the Langlands conjugacy class of $\pi_i$ at unramified $v$ by $\lambda_{\pi_i}(v)$. For any $\alpha\in\Bbb{R}$, we define
$$
\mathcal{S}_\alpha= \mathcal{S}_\alpha(\pi_1,\pi_2) = \{\text{$v$ unramified for both $\pi_1$ and $\pi_2$} \mid \lambda_{\pi_1}(v)\neq e^{i\alpha} \lambda_{\pi_2}(v) \}.
$$
We  let $\underline{\delta}(\mathcal{S}_\alpha)$ denote the lower Dirichlet density of $\mathcal{S}_\alpha$ and let $\chi_{\mathcal{S}_\alpha}$ denote the indicator function of $\mathcal{S}_\alpha$.

Writing $a_v=\lambda_{\pi_1}(v)$ and  $b_v=\lambda_{\pi_2}(v)$, in light of the argument used by Walji \cite{Wa14} (who considered the case that $\alpha=0$), to study the lower bound of $\mathcal{S}_\alpha$, we shall apply the following consequence of the  Cauchy-Schwarz inequality:
\begin{equation}\label{C-S-ineq}
\sum_v \frac{|a_v -e^{i\alpha}b_v |^2 }{Nv^{s}}= \sum_v \frac{|a_v -e^{i\alpha}b_v |^2 \chi_{S_\alpha}(v)}{Nv^{s}}
\le \Big( \sum_v \frac{|a_v -e^{i\alpha}b_v |^4 }{Nv^{s}}\Big)^{\frac{1}{2}} \Big( \sum_{v\in S_\alpha} \frac{1}{Nv^{s}} \Big)^{\frac{1}{2}},
\end{equation} 
where, as later, the sums are over  unramified primes $v$ for both $\pi_1$ and $\pi_2$. Recalling that
$$
\lD( S_\alpha)= \liminf_{s\rightarrow 1^+} \frac{\sum_{v\in S_\alpha} \frac{1}{Nv^{s}}}{ \log(\frac{1}{s-1})},
$$
to obtain a lower bound for $\underline{\delta}(\mathcal{S}_\alpha)$, we shall  analyse the asymptotic behaviours of sums in \eqref{C-S-ineq} as real $s\rightarrow 1^+$.

  From the identities $|a_v -e^{i\alpha}b_v |^2 = a_v \bar{a}_v-e^{-i\alpha}a_v \bar{b}_v -e^{i\alpha} \bar{a}_v b_v  + b_v \bar{b}_v$ and 
\begin{align}\label{4th-power}
 \begin{split}
|a_v -e^{i\alpha}b_v |^4 &= a_v^2 \bar{a}_v^{2} + e^{-2i\alpha} a_v^2 \bar{b}_v^{2} + e^{2i\alpha}\bar{a}_v^{2}b_v^2    + b_v^2 \bar{b}_v^{2} + 4 a_v \bar{a}_v b_v \bar{b}_v \\
& -2  e^{-i\alpha}  a_v^2 \bar{a}_v \bar{b}_v -2  e^{i\alpha}  a_v \bar{a}_v^2 b_v  
 -2  e^{-i\alpha}  a_v b_v \bar{b}_v^2 -  2  e^{i\alpha}  \bar{a}_v b_v^2  \bar{b}_v,
 \end{split}
\end{align}
it is not hard to see that to bound the sums in \eqref{C-S-ineq} as real $s\rightarrow 1^+$, it suffices to study the asymptotic behaviours of the Dirichlet series
$$
\mathcal{D}(s;i,j,k,l) = \sum_v \frac{a_v^i \bar{a}_v^{j} b_v^k \bar{b}_v^{l}}{Nv^s}
$$
as real $s\rightarrow 1^+$. Indeed, for example, 
$$
\log (L(s, \pi_1\boxtimes \bar{\pi}_1 \times \pi_2\boxtimes \bar{\pi}_2)) = \mathcal{D}(s;1,1,1,1) +O(1),
$$
where the big-O term is contributed by prime powers of $v$ (we note that $\lambda_{\pi_1}(v^k)$ and $\lambda_{\pi_2}(v^k)$ can be controlled by the bound established by Blomer and Brumley \cite{BlBr11}), and thus
$$
\lim_{s\rightarrow 1^+} \frac{ \mathcal{D}(s;1,1,1,1)}{ \log(\frac{1}{s-1})} =  \lim_{s\rightarrow 1^+} \frac{\log (L(s, \pi_1\boxtimes \bar{\pi}_1 \times \pi_2\boxtimes \bar{\pi}_2))}{ \log(\frac{1}{s-1})} =\delta_{1,1,1,1}, 
$$
where $\delta_{1,1,1,1}$ denotes the order of the pole of $L(s,\pi_1\boxtimes \bar{\pi}_1 \times \pi_2\boxtimes \bar{\pi}_2)$ at $s=1$.  Hence, to prove our main result, it is sufficient to bound the orders of the poles of $L$-functions involved at $s=1$. We shall use this argument throughout Sections  \ref{Dih}, \ref{one-non-dih}, and \ref{proof-lD_S*}.

To prove Theorem \ref{main-thm-non-twist-equiv}, we borrow the following estimates  from \cite[Sec. 2 and p. 4999]{Wa14}.
 
\begin{proposition} 
In the notation as above, let $\pi_1$ and $\pi_2$ be distinct. For $0\le i,j,k,l\le 2$, let 
$$
\delta_{i,j,k,l}= \lim_{s\rightarrow 1^+} \frac{ \mathcal{D}(s;i,j,k,l)}{ \log(\frac{1}{s-1})}.
$$

If $\pi_1$ is not dihedral but $\pi_2$ is dihedral, then one has  
\begin{equation}\label{n-d&d}
\delta_{i,j,k,l}  \le 
\begin{cases}
 0 & \text{if $(i,j,k,l)\in\{ (2, 1, 0, 1), (0, 1, 2, 1), (1, 2, 1, 0), (1, 0, 1, 2)\}$;}\\
 1 & \text{if $(i,j,k,l)= (1,1,1,1)$;}\\
 2 & \text{if $(i,j,k,l)\in\{ (2, 2, 0, 0), (2, 0, 0, 2), (0, 2, 2, 0)\}$;}\\
 4 & \text{if $(i,j,k,l)=(0, 0, 2, 2)$.}
\end{cases}
\end{equation}

If both $\pi_1$ and $\pi_2$  are   non-dihedral and non-tetrahedral,\footnote{A cuspidal automorphic representation $\pi$ for $\GL_2(\Bbb{A}_F )$ is called tetrahedral if it is non-dihedral and its symmetric square $\Sym^2\pi$ admits a non-trivial self-twist by a (cubic) Hecke character.  By the work of Kim and  Shahidi  \cite{KS00, KS02},  $\Sym^3\pi$ is cuspidal if $\pi$ is non-tetrahedral.}   then one has  
\begin{equation}\label{both-non-te}
\delta_{i,j,k,l} \le
 \begin{cases} 
0 & \text{if $(i,j,k,l)\in\{ (2, 1, 0, 1), (0, 1, 2, 1), (1, 2, 1, 0), (1, 0, 1, 2)\}$;}\\
2& \text{if $(i,j,k,l)\in\{ (1, 1, 1, 1), (2, 2, 0, 0), (0, 0, 2, 2), (2, 0, 0, 2), (0, 2, 2, 0)\}$.}
 \end{cases}
\end{equation}

If both $\pi_1$ and $\pi_2$  are  non-dihedral, and at least one of $\pi_1$ and $\pi_2$ is tetrahedral, then one has  
\begin{equation}\label{one-te}
\begin{cases}
\delta_{i,j,k,l}= \kappa_1 & \text{if $(i,j,k,l)\in\{ (2, 1, 0, 1), (0, 1, 2, 1), (1, 2, 1, 0), (1, 0, 1, 2)\}$;}\\
\delta_{i,j,k,l}\le 2& \text{otherwise,}
\end{cases}
\end{equation}
where $\kappa_1$ is 1  if $\pi_1\simeq \pi_2 \otimes \nu$ for some cubic Hecke character $\nu$ and 0 otherwise.

If  $\pi_1$ and $\pi_2$ are  non-dihedral and non-twist-equivalent, then one has  
\begin{equation}\label{n-t-e}
\begin{cases}
\delta_{i,j,k,l}  = 0 & \text{if $(i,j,k,l)\in\{ (2, 1, 0, 1), (0, 1, 2, 1), (1, 2, 1, 0), (1, 0, 1, 2)\}$;}\\
\delta_{i,j,k,l} \le 1 & \text{if $(i,j,k,l)=(1,1,1,1)$;} \\
\delta_{i,j,k,l} \le  2& \text{if $(i,j,k,l)\in\{  (2, 2, 0, 0), (0, 0, 2, 2)\}$;}\\
\delta_{i,j,k,l} =\kappa_2& \text{if $(i,j,k,l)\in\{ (2, 0, 0, 2), (0, 2, 2, 0)\}$,}
\end{cases}
\end{equation}
where $\kappa_2$ is 1  if $\omega_1= \omega_2$ and 0 otherwise.

\end{proposition} 

\begin{remark}
Let $\pi$ and $\pi'$ be  cuspidal automorphic representations, with unitary central characters, for $\GL_n(\Bbb{A}_F )$ and $\GL_m(\Bbb{A}_F )$, respectively. By the theory of Rankin-Selberg $L$-functions, we know that the poles of $L(s,\pi\times \pi')$ and $L(s,\bar{\pi}\times \bar{\pi}')$ at $s=1$ must have the same order. Consequently,  the poles of $\mathcal{D}(s;2,0,0,2)$ and $\mathcal{D}(s;0,2,2,0)$ at $s=1$ are of the same order, and thus $\delta_{2,0,0,2}=\delta_{0,2,2,0}$. Similarly, $\delta_{2, 1, 0, 1}= \delta_{1, 2, 1, 0}$ and $ \delta_{0, 1, 2, 1} = \delta_{1, 0, 1, 2}$. We shall use this fact throughout our discussion.
\end{remark}

\subsection{Joint Sato-Tate distribution and Selberg polynomials}\label{jSTd-Sp}

To adapt the strategies developed in  Murty-Pujahari \cite{MP17} and   Murty-Rajan \cite{MR96}, we shall recall the joint Sato-Tate distribution and some basic properties of Selberg polynomials.

Let $F$ be a totally real number field.  For each $i$, let $\pi_i$ be a cuspidal automorphic representation  corresponding to a non-CM Hilbert newform   of weights $k_{i,j}\ge 2$ (at all infinite primes $v_j$ of $F$) and with trivial nebentypus. For each $i$, we write
$$
\lambda_{\pi_i}(v)=2\cos  \theta_{i,v}
$$
for some $\theta_{i,v}\in [0,\pi]$. (Recall that by the work of  Blasius \cite{Bl06}, the Ramanujan-Petersson conjecture holds for each $\pi_i$.)

We recall that the $n$-th  Chebyshev polynomial $U_n(x)$  (of the second type) is defined by
$$
U_n(\cos\theta)= \frac{\sin( (n+1)\theta)}{\sin\theta}.
$$
(Note that $U_0(\cos\theta)\equiv 1$ and $U_1(\cos\theta)=2 \cos\theta$.)

We will require the following version of the joint Sato-Tate distribution, which is a consequence of the work of Barnet-Lamb \emph{et al.} \cite{BLGG11, BGGT11}  (see \cite[Theorem 1.1 and Sec. 3]{PJW19} for more details).

\begin{proposition}[Barnet-Lamb \emph{et al.}]\label{uncond-ST}
In the notation as above, for any $m_1,m_2\in\Bbb{N}$,  one has
\begin{align*}
\sum_{  Nv\le x  } U_{m_1}(\cos\theta_{1,v})U_{m_2}(\cos\theta_{2,v})  =  o(\pi_F(x) ),
\end{align*}
where $\pi_F(x)$ denotes the number of primes $v$ of $F$ such that $Nv\le x$.
\end{proposition}

We shall further require the following effective versions of the joint Sato-Tate distribution, proved by Thorner \cite[Proposition 2.2]{Th20} and the author \cite{PJW19}.

\begin{proposition}[Thorner]\label{Th2} 
For each $i$, let $\pi_i$ be a cuspidal automorphic representation  corresponding to a non-CM  newform  in $S_{k_i}^{\mathrm{new}}(\Gamma_0(q_i))$  with trivial nebentypus.  Assume that all the symmetric powers $\Sym^{m_1}\pi_1$ and $\Sym^{m_2}\pi_2$ are automorphic. Then there exist positive constants $c_1$, $c_2$, $c_3$, $c_4$ and $c_5$  such that  for any
$$
1\le m_1, m_2\le M \le c_1  \sqrt{\log \log x}/ \log (k_1q_1k_2q_2 \log \log x),
$$
one has
\begin{align*}
\sum_{  p\le x  }  & U_{m_1}(\cos\theta_{1,p})U_{m_2}(\cos\theta_{2,p})  \ll \pi(x)\exp\Big(\frac{-c_2 \log x}{(k_1q_1k_2q_2  M)^{ c_3 M^2}} \Big)    \\
&  +\pi(x)(m_1 m_2)^2  \Big( x^{\frac{-1}{c_4 M^2}} +\exp\Big(\frac{-c_5\log x}{ M^2\log (k_1q_1k_2q_2 M) } \Big)+\exp\Big(\frac{ -c_5 \sqrt{\log x}  }{ M} \Big)   \Big).
\end{align*}
\end{proposition}

\begin{proposition}\label{PJW} \cite[p. 287]{PJW19}
In the notation of Proposition \ref{Th2}, assume that  the symmetric powers $\Sym^{m_1}\pi_1$ and $\Sym^{m_2}\pi_2$ are automorphic. If the Rankin-Selberg $L$-function $L(s,\Sym^{m_1}\pi_1 \times \Sym^{m_2}\pi_2)$ satisfies the generalised Riemann hypothesis, then
\begin{align*}
\sum_{   p\le x  }  & U_{m_1}(\cos\theta_{1,p})U_{m_2}(\cos\theta_{2,p})  \ll m_1m_2 x^{1/2}\log ( (k_1q_1k_2q_2 )(m_1+m_2) x) .
\end{align*}
\end{proposition}

\begin{remark}
We shall note that the automorphy assumption of  $\Sym^{m_1}\pi_1$ and $\Sym^{m_2}\pi_2$ in Propositions \ref{Th2} and  \ref{PJW} can be removed by invoking the recent work of Newton and Thorne \cite{NT20}.
\end{remark}

To end this section, we review some basic properties of Selberg polynomials. We begin by recalling that for any integer $M\ge 1$, the Vaaler polynomial $V_M(x)$ is defined by
\begin{align*}
V_M(x)&=\frac{1}{M+1}\sum_{k=1}^M \Big(\frac{k}{M+1}-\frac{1}{2}\Big)\Delta_{M+1}\Big(x-\frac{k}{M+1}\Big)\\
&+\frac{1}{2\pi(M+1)}\sin(2\pi(M+1 )x) -\frac{1}{2\pi}\Delta_{M+1}(x) \sin (2\pi x),
\end{align*}
where $\Delta_{M}(x)=\frac{1}{M}(\frac{\sin(\pi Mx)}{\sin(\pi x)})^2$ is the Fej\'er kernel (see, e.g., \cite[Sec. 1.2, Eq. (16) and (17)]{Mon94}). Following \cite[Sec. 1.2, Eq. ($21^+$)]{Mon94}, for any subinterval $J=[0,\delta]\subseteq [0,1]$ and integer $M\ge 1$, we define the Selberg polynomial 
\begin{equation}\label{S+}
S^+_{J,M}(x) = \delta + B_M(x - \delta) + B_M(-x),
\end{equation}
where $B_M(x)$ is the Beurling polynomial as defined in \cite[Sec. 1.2, Eq. (20)]{Mon94}, namely,
$$
B_M(x) = V_M(x) +\frac{1}{2(M+1)} \Delta_{M+1}(x).
$$
We recall that
$$
\chi_{J}(x)\le S^+_{J,M}(x),
$$
where $\chi_J$ is the indicator function of $J$. (In what follows, we shall regard both $S^+_{J,M}(x)$ and $\chi_{J}(x)$ as functions of period 1 defined over $\Bbb{R}$.) Moreover,  writing the Fourier expansion of $S^+_{J,M}(x)$ as
$$
S^+_{J,M}(x)= \sum_{n=-\infty}^{\infty}   \hat{S}^+_{J,M}(n) e^{2\pi i n x},
$$
we know that 
$$
 \hat{S}^+_{J,M}(n)=
\begin{cases}
\delta +\frac{1}{M+1} & \text{if $n=0$;}\\
0& \text{if $|n|>M$;}
\end{cases}
$$
also, for $1\le |n| \le M$, 
$$
| \hat{S}^+_{J,M}(n) | \le \frac{1}{M+1} +\min\Big\{\delta, \frac{1}{\pi |n|} \Big\}
$$
(see \cite[pp. 6-8]{Mon94}). From the definition of Fourier transform, it follows directly that
$$
\hat{S}^+_{J,M}(n) +\hat{S}^+_{J,M}(-n) =2\Re(\hat{S}^+_{J,M}(n) ),
$$
and thus
\begin{align*}
S^+_{J,M}(x)+S^+_{J,M}(-x) 
&= 2\delta +\frac{2}{M+1}+  2\sum_{0<|n|\le M} \hat{S}^+_{J,M}(n) \cos(2\pi nx)\\
&=  2\delta +\frac{2}{M+1}+  4\sum_{n=1}^M \Re(\hat{S}^+_{J,M}(n) ) \cos(2\pi nx).
\end{align*}
To summarise, we have the following proposition.
\begin{proposition}\label{Sel-poly}
In the notation as above,  for any subinterval $J=[0,\delta]\subseteq [0,1]$ and integer $M\ge 1$, one has
$$
\mathcal{I}_{\delta}(\theta):=\frac{1}{2}\Big( \chi_{J}\Big(\frac{\theta}{2\pi}\Big)+\chi_{J}\Big(-\frac{\theta}{2\pi}\Big) \Big)\le   \delta +\frac{1}{M+1}+  2\sum_{n=1}^M \Re(\hat{S}^+_{J,M}(n) ) \cos(n\theta),
$$
where for $1\le n \le M$,
$$
|\Re(\hat{S}^+_{J,M}(n) )| \le  \frac{1}{M+1} +\min\Big\{\delta, \frac{1}{\pi n} \Big\}.
$$
\end{proposition}

\section{Proofs of Theorems \ref{Main-dist} and \ref{main-thm-non-twist-equiv}, part I: both $\pi_1$ and $\pi_2$ are dihedral}\label{Dih}

In this section, we shall prove Theorems \ref{Main-dist} and \ref{main-thm-non-twist-equiv} for the case that both $\pi_1$ and $\pi_2$ are dihedral. We will emphasise the case that $\pi_1$ and $\pi_2$ are twist-equivalent. (The case that $\alpha=0$ was not treated in \cite{Wa14} as the bound established by Ramakrishnan \cite{Rama94} is already sharp.)

Recall that every dihedral cuspidal automorphic representation $\pi$ for $\GL_2(\Bbb{A}_F )$  can be induced from a Hecke character $\psi$ of $K$ for some quadratic extension $K$ of $F$. For such an instance, we shall write $\pi = I_K^F(\psi)$.  Following Walji \cite[p. 4995]{Wa14}, we  say that a dihedral cuspidal automorphic representation $\pi$ for $\GL_2(\Bbb{A}_F )$  has property $\mathcal{P}$ if   $\pi=I_K^F(\psi)$, and the Hecke character $\psi/\psi^{\tau}$ is invariant under $\tau$, the non-trivial element of $\Gal(K/F)$. We also recall that $\pi_1\boxtimes \pi_2$ is cuspidal if and only if $\pi_1$ and $\pi_2$ cannot be induced from the same quadratic extension of $F$.

In this section, we let $\pi_1$ and $\pi_2$ be dihedral cuspidal automorphic representations for $\GL_2(\Bbb{A}_F )$ induced from $\nu$ and $\mu$ of quadratic extensions $K_1$ and $K_2$ of $F$, respectively. In addition, we let $\chi_i$ be the Hecke character associated to $K_i/F$ and let $\tau_i$ be the non-trivial element of $\Gal(K_i/F)$. Besides, if $\pi_1$ and $\pi_2$ can be induced from the same quadratic extension $K$ of $F$, then we shall choose $K_1=K_2=K$ and define $\chi=\chi_1=\chi_2$ and $\tau=\tau_1=\tau_2$.

 We shall argue according to whether $\pi_i$ has property $\mathcal{P}$  for each $i$.

\subsection{Exactly one of $\pi_1$ and $\pi_2$ has property $\mathcal{P}$}
 Assume that exactly one of the dihedral automorphic representations $\pi_1$ and $\pi_2$ satisfies property $\mathcal{P}$. Without loss of generality, we consider the case that $\pi_1 =I_K^F(\nu)$ has property $\mathcal{P}$ (i.e., $(\nu/\nu^{\tau_1})^{\tau_1} = \nu/\nu^{\tau_1}$, where $\tau_1$ is the non-trivial element of $\Gal(K_1/F)$). 
 
\subsubsection{$\pi_1$ and $\pi_2$ cannot be induced from the same quadratic extension of $F$}\label{OX-not-sameK}
We first consider that case that $\pi_1$ and $\pi_2$ cannot be induced from the same quadratic extension of $F$. (This case was not discussed in \cite[p. 4998]{Wa14}. Nonetheless,  the bound for this case is better than the one for the case that  $\pi_1$ and $\pi_2$ can be induced from the same quadratic extension.)

Since $\pi_1$ has property $\mathcal{P}$, we have
\begin{equation}\label{decomp-pi1-P}
\pi_1\boxtimes\bar{\pi}_1 \simeq 1 \boxplus \chi_1 \boxplus \nu/\nu^{\tau_1} \boxplus (\nu/\nu^{\tau_1})\chi_1 ,
\end{equation}
where $\chi_1$ is the Hecke character associated to $K_1/F$. Also, as $\pi_1\boxtimes\bar{\pi}_1\simeq 1\boxplus \Ad(\pi_1)$, we see that
\begin{equation}\label{Ad-decomp-pi1-P}
 \Ad(\pi_1) \simeq  \chi_1 \boxplus \nu/\nu^{\tau_1} \boxplus  (\nu/\nu^{\tau_1})\chi_1.
\end{equation}
On the other hand, since  $\pi_2$ does not have property $\mathcal{P}$, 
\begin{equation}\label{decomp-pi2-NP}
\pi_2\boxtimes\bar{\pi}_2 \simeq 1 \boxplus \chi_2 \boxplus I_{K_2}^F(\mu/\mu^{\tau_2}),
\end{equation}
where $\chi_2$ is the Hecke character associated to $K_2/F$, and $I_{K_2}^F(\mu/\mu^{\tau_2})$ is cuspidal. Thus,
\begin{equation}\label{Ad-decomp-pi2-NP}
 \Ad(\pi_2) \simeq  \chi_2  \boxplus I_{K_2}^F(\mu/\mu^{\tau_2}).
\end{equation}
Since each $\Ad(\pi_i)$ is self-dual, we know that $L(s,\Ad(\pi_1) \times   \Ad(\pi_1))$ has a pole of order three at $s=1$ and that $L(s,\Ad(\pi_2) \times   \Ad(\pi_2))$ has a pole of order two at $s=1$. (Note that $ \nu/\nu^{\tau_1} \not\simeq \chi_1 $; otherwise, $\pi_1\simeq \nu\boxplus \nu\chi_1$, which contradicts to the cuspidality of $\pi_1$.) Moreover, in this case, $\pi_1\boxtimes \pi_2$ is cuspidal. Thus, from the identity
$$
L(s, ( 1\boxplus\Ad(\pi_1) )\times  (1 \boxplus \Ad(\pi_2) )) = L(s, \pi_1 \boxtimes \bar{\pi}_1 \times \pi_2 \boxtimes \bar{\pi}_2 ) = L(s, \pi_1 \boxtimes \pi_2 \times \bar{\pi}_1\boxtimes \bar{\pi}_2),
$$
we see that $L(s,  \Ad(\pi_1) \times   \Ad(\pi_2) )$ is holomorphic  at $s=1$.

Since  the Ramanujan-Petersson conjecture holds for $\pi_1$ and $\pi_2$ (as they are dihedral), we have $|\lambda_{\pi_i}(v)|^2\le 4$. Recalling that for each $i$,
\begin{equation}\label{Ad-square}
\lambda_{ \Ad(\pi_i)}(v)=  |\lambda_{\pi_i}(v)|^2 -1,
\end{equation}
we derive
$$
|\lambda_{ \Ad(\pi_1)}(v) - \lambda_{\Ad(\pi_2)}(v) | =| |\lambda_{\pi_1}(v)|^2 -1 - |\lambda_{\pi_2}(v)|^2 +1|  \le 4
$$ 
and thus 
\begin{equation}\label{Ad-ineq}
\sum_{v} \frac{|\lambda_{ \Ad(\pi_1)}(v) - \lambda_{\Ad(\pi_2)}(v) |^2}{Nv^s}=\sum_{v} \frac{|\lambda_{ \Ad(\pi_1)}(v) - \lambda_{\Ad(\pi_2)}(v) |^2 \chi_{\Ad}(v)}{Nv^s} \le  16  \sum_{v\in \mathcal{S}_{\Ad}} \frac{1}{Nv^s},
\end{equation}
where
\begin{equation}\label{SAd-def}
\mathcal{S}_{\Ad}= \mathcal{S}_{\Ad}(\pi_1,\pi_2) = \{\text{$v$ unramified for both $\pi_1$ and $\pi_2$} \mid \lambda_{\Ad(\pi_1)}(v)\neq  \lambda_{\Ad(\pi_2)}(v) \} .
\end{equation}
Note that as each  $\Ad(\pi_i)$ is self-dual, for any unramified $v$,
$$
|\lambda_{ \Ad(\pi_1)}(v) - \lambda_{\Ad(\pi_2)}(v) |^2 = \lambda_{ \Ad(\pi_1)\times \Ad(\pi_1) }(v) -2 \lambda_{ \Ad(\pi_1)\times \Ad(\pi_2)}(v) + \lambda_{ \Ad(\pi_2)\times\Ad(\pi_2)}(v).
$$

Now, dividing both sides of \eqref{Ad-ineq} by $\log(\frac{1}{s-1})$ and making real $s\rightarrow 1^+$, 
we deduce that $5= 3+2 \le  16 \lD(\mathcal{S}_{\Ad})$. By \eqref{Ad-square}, we know that if $\lambda_{ \Ad(\pi_1)}(v)\neq \lambda_{ \Ad(\pi_2)}(v)$, then $  |\lambda_{\pi_1}(v)|\neq   |\lambda_{\pi_2}(v)|$. Thus, for any $\alpha$, $\mathcal{S}_{\Ad} \subseteq  \mcS_\alpha$, and we obtain
$$
 \frac{5}{16}\le     \lD(\mathcal{S}_{\Ad}) \le \lD(\mcS_\alpha).
$$ 
 
\subsubsection{$\pi_1$ and $\pi_2$ can be induced from the same quadratic extension of $F$} 
 \label{OX-sameK}
Suppose that  $\pi_1$ and $\pi_2$ can be induced from the same quadratic extension $K$ of $F$.   As argued in \cite[p. 4998]{Wa14}, for any prime $\omega$ of $K$, we have $\nu(\omega) = \pm \nu^{\tau} (\omega)$, and thus for any prime $v$ of $F$, $|\lambda_{\pi_1}(v)|=|\lambda_{I_K^F(\nu)}(v)|$ equals  either 0 or 2. 

Write $\pi_2= I_K^F(\mu)$ as before. As $(\mu/\mu^{\tau})^{\tau} \neq \mu/\mu^{\tau}$,   $\mu^2/(\mu^{\tau})^2$ is non-trivial, and thus there is a set $S$  of density $1/4$ consisting of primes $v$ of $F$  that split in $K$ such that $\mu^2/(\mu^{\tau})^2(\omega)\neq 1$ for $\omega\mid v$. Hence, if $\omega\mid v$ for some $v\in S$, then $\mu(\omega)\neq \mu^{\tau}(\omega)$, which implies that  $|\lambda_{\pi_2}(v)|=|\lambda_{I_K^F(\mu)}(v)|$ is not equal to  0 nor 2. Therefore, for any $v\in S$,  $|\lambda_{\pi_1}(v)|\neq |\lambda_{\pi_2}(v)|$
and thus $|\lambda_{ \Ad(\pi_1)}(v)|\neq|\lambda_{\Ad(\pi_2)}(v) |$ for any $v\in S$. In particular, for any $v\in S$ and any $\alpha\in \Bbb{R}$, 
$$
\lambda_{\pi_1}(v)\neq e^{i\alpha} \lambda_{\pi_2}(v).
$$
From the above discussion, we see that 
$$
\frac{1}{4}\le  \lD(S)= \lD(\mcS_{\Ad}) \le  \lD(\mcS_\alpha).
$$ 
whenever exactly one of $\pi_1$ and $\pi_2$ satisfies property $\mathcal{P}$.

\begin{remark}
If we argue as in the previous section, as $L(s,\Ad(\pi_1)\times \Ad(\pi_2))$ has a pole of order one at $s=1$, we will have a slightly weaker lower bound:
$$
\lD(\mcS_{\Ad})\ge \frac{3-2+2}{16}= \frac{3}{16}
$$
for the case that $\pi_1$ and $\pi_2$ can be induced from the same quadratic extension of $F$.
\end{remark}

\subsection{Both $\pi_1$ and $\pi_2$ do not have property $\mathcal{P}$}\label{XX}

\subsubsection{$\pi_1$ and $\pi_2$ are twist-equivalent} Assume that both $\pi_1$ and $\pi_2$ do not have property $\mathcal{P}$ and that  $\pi_1$ and $\pi_2$ are twist-equivalent. Note that
\begin{equation}\label{decomp-pi1-NP}
\pi_1\boxtimes\bar{\pi}_1 \simeq 1 \boxplus \chi_1 \boxplus I_{K_1}^F(\nu/\nu^{\tau_1}),
\end{equation}
where $\chi_1$ is the quadratic Hecke character associated to $K_1/F$, and  $I_{K_1}^F(\nu/\nu^{\tau_1})$ is cuspidal. From  \eqref{decomp-pi1-NP}, it follows that
$$
L(s, \pi_1\boxtimes \bar{\pi}_1 \times  \pi_1\boxtimes \bar{\pi}_1 )= L(s,(1 \boxplus \chi_1 \boxplus I_{K_1}^F(\nu/\nu^{\tau_1}))
\times ( \overline{ 1 \boxplus \chi_1 \boxplus I_{K_1}^F(\nu/\nu^{\tau_1})}  ))
$$
has a pole of order three at $s=1$. Also, as $\Ad(\pi_1)\simeq \Ad(\pi_2)$,  $\pi_2\boxtimes \bar{\pi}_2 \simeq \pi_1\boxtimes \bar{\pi}_1 $, and thus
$$
L(s, \pi_2\boxtimes \bar{\pi}_2 \times  \pi_2\boxtimes \bar{\pi}_2 )= L(s, \pi_1\boxtimes \bar{\pi}_1 \times  \pi_2\boxtimes \bar{\pi}_2 )= L(s, \pi_1\boxtimes \bar{\pi}_1 \times  \pi_1\boxtimes \bar{\pi}_1 )
$$
has a pole of order  three at $s=1$. 
 
Now, we consider the $L$-function $L(s, \pi_1\boxtimes \pi_1 \times  \bar{ \pi}_1\boxtimes \bar{ \pi}_2 ) $. Writing $\pi_2 \simeq \pi_1\otimes \chi $ for some (non-trivial) id\`ele class character $\chi$, we deduce
\begin{align*}
L(s, \pi_1\boxtimes \pi_1 \times  \bar{ \pi}_1\boxtimes \bar{ \pi}_2 ) 
&= L(s, \pi_1 \boxtimes   \bar{ \pi}_1 \times  \pi_1 \boxtimes  ( \bar{ \pi}_1\otimes \bar{\chi} ) )\\
&=L(s,(1 \boxplus \chi_1 \boxplus I_{K_1}^F(\nu/\nu^{\tau_1}))\times (\bar{ \chi} \boxplus \chi_1\bar{ \chi} \boxplus I_{K_1}^F(\nu/\nu^{\tau_1})\otimes \bar{ \chi}))
\end{align*}
has a pole of order at most three at $s=1$. Similarly, $L(s, \pi_1\boxtimes \pi_1 \times   \bar{ \pi}_2\boxtimes \bar{ \pi}_2 ) 
= L(s, \pi_1 \boxtimes   \bar{ \pi}_1 \times  \pi_1 \boxtimes  ( \bar{ \pi}_1\otimes \bar{\chi}^2 ) )$ has a pole of order at most three at $s=1$.

As  $\delta_{i,j,k,l}\le 3$ for all cases, we derive
$$
2^2\le  \begin{cases} 
(3+ 6\cos(2\alpha) + 3 +12  ) \lD(\mathcal{S}_\alpha)  & \text{if $\cos(2\alpha)\ge 0$ and $\cos\alpha\ge 0$;}\\
(3+ 6\cos(2\alpha)  + 3 +12  - 24\cos\alpha) \lD(\mathcal{S}_\alpha)  & \text{if $\cos(2\alpha)\ge 0$ and $\cos\alpha\le 0$;}\\
(3+ 3 +12  ) \lD(\mathcal{S}_\alpha)  & \text{if $\cos(2\alpha)\le 0$ and $\cos\alpha\ge 0$;}\\
(3 + 3 +12 - 24\cos\alpha ) \lD(\mathcal{S}_\alpha)  & \text{if $\cos(2\alpha)\le 0$ and $\cos\alpha\le 0$.}
 \end{cases}
$$

\subsubsection{$\pi_1$ and $\pi_2$ are not twist-equivalent}\label{1/8} Now, we consider the situation that $\pi_1$ and $\pi_2$ do not have property $\mathcal{P}$, and  $\pi_1$ and $\pi_2$ are not twist-equivalent. Similar to \eqref{Ad-decomp-pi2-NP}, by \eqref{decomp-pi1-NP}, we have
\begin{equation}\label{Ad-decomp-pi1-NP}
 \Ad(\pi_1) \simeq  \chi_1  \boxplus I_{K_1}^F(\nu/\nu^{\tau_1}). 
\end{equation}
Recalling that $I_{K_1}^F(\nu/\nu^{\tau_1})$ and $I_{K_2}^F(\mu/\mu^{\tau_2})$ are cuspidal, we note that  as $\Ad(\pi_1)\not\simeq \Ad(\pi_2)$ and $\Ad(\pi_2)$ is self-dual,   $\Ad(\pi_1)\not\simeq  \overline{\Ad(\pi_2)}$ and thus either $\bar{\chi}_2\not\simeq \chi_1$ or $\overline{I_{K_2}^F(\mu/\mu^{\tau_2})}\not\simeq I_{K_1}^F(\nu/\nu^{\tau_1})$.  Hence, by \eqref{Ad-decomp-pi2-NP} and \eqref{Ad-decomp-pi1-NP}, we see that
$$
L(s,\Ad(\pi_1)\times \Ad(\pi_2))= L(s, (\chi_1 \boxplus I_{K_1}^F(\nu/\nu^{\tau_1}))\times (\chi_2 \boxplus I_{K_2}^F(\mu/\mu^{\tau_2})))
$$
has a pole of order at most one at $s=1$. Also, from \eqref{Ad-decomp-pi2-NP} and \eqref{Ad-decomp-pi1-NP}, it follows immediately that for each $i$, $L(s,\Ad(\pi_i)\times \Ad(\pi_i))= L(s,\Ad(\pi_i)\times  \overline{\Ad(\pi_i)})$ has a pole of order two at $s=1$. Therefore,  by the estimate \eqref{Ad-ineq}, we obtain
$$
2= 2-2+2 \le  16 \lD(\mathcal{S}_{\Ad}),
$$
where $\mathcal{S}_{\Ad}$ is defined as in \eqref{SAd-def}. Since for any $\alpha$, $\mathcal{S}_{\Ad} \subseteq  \mcS_\alpha$,  we obtain
$$
 \frac{1}{8}\le     \lD(\mathcal{S}_{\Ad}) \le \lD(\mcS_\alpha).
$$

This can be further improved for $\mcS_\alpha$ by  Walji's argument.  It was shown in \cite[pp. 4995-4996]{Wa14} that 
$$
\delta_{i,j,k,l}  \le 
\begin{cases}
 1 & \text{if $(i,j,k,l)\in\{ (2, 1, 0, 1), (0, 1, 2, 1), (1, 2, 1, 0), (1, 0, 1, 2)\}$;}\\
 2 & \text{if $(i,j,k,l)\in\{ (1,1,1,1), (2, 0, 0, 2), (0, 2, 2, 0) \}$;}\\
 3 & \text{if $(i,j,k,l)\in\{(2, 2, 0, 0), (0, 0, 2, 2 \}$.}
\end{cases}
$$
Hence,  we have
$$
2^2 \le
\begin{cases}
 (3 + 4 \cos(2\alpha) +3  + 8 )     \lD(\mcS_\alpha) & \text{if $\cos(2\alpha)\ge 0$ and $\cos\alpha\ge 0$;}\\
 (3 + 4 \cos(2\alpha) +3  + 8 -8 \cos \alpha)  \lD(\mcS_\alpha)  & \text{if  $\cos(2\alpha)\ge 0$ and $\cos\alpha \le 0$;}\\
   (3  +3  + 8 )  \lD(\mcS_\alpha)  & \text{if  $\cos(2\alpha)\le 0$ and $\cos\alpha \ge 0$;}\\
  (3   +3  + 8 -8 \cos \alpha)  \lD(\mcS_\alpha)  & \text{if  $\cos(2\alpha)\le 0$ and $\cos\alpha \le 0$.}
\end{cases}
$$
Thus, for any $\alpha\in\Bbb{R}$, $\lD(\mcS_\alpha) \ge  \frac{2}{13}$ (which is $>\frac{1}{8}$). (We note that although our method gives a worse bound than the one given by Walji's argument, it will be useful in Section \ref{proof-lD_S*}.)

\subsection{Both $\pi_1$ and $\pi_2$ have property $\mathcal{P}$}

\subsubsection{$\pi_1$ and $\pi_2$ are twist-equivalent} We move on to the case of twist-equivalent $\pi_1$ and $\pi_2$ that satisfy property $\mathcal{P}$. We recall that in Section \ref{OX-not-sameK}, we showed that
$$
1 \boxplus \Ad(\pi_1) \simeq
 \pi_1\boxtimes\bar{\pi}_1 \simeq 1 \boxplus \chi_1 \boxplus \nu/\nu^{\tau_1} \boxtimes  (\nu/\nu^{\tau_1})\chi_1 
$$
and that
$L(s,\Ad(\pi_1)\times \Ad(\pi_1))$ has a pole of order three at $s=1$. Thus,
$$
L(s, \pi_1\boxtimes\bar{\pi}_1 \times \pi_1\boxtimes\bar{\pi}_1)=L(s,1)L(s,\Ad(\pi_1))^2L(s,\Ad(\pi_1)\times \Ad(\pi_1) )
$$
has a pole of order four at $s=1$. Since  $\pi_1 \simeq \pi_2 \otimes \chi$
for some (non-trivial) Hecke character, we know that
$$
L(s, \pi_1\boxtimes\bar{\pi}_1 \times \pi_1\boxtimes\bar{\pi}_1)= L(s, \pi_1\boxtimes\bar{\pi}_1 \times \pi_2\boxtimes\bar{\pi}_2) =  L(s, \pi_2\boxtimes\bar{\pi}_2 \times \pi_2\boxtimes\bar{\pi}_2)
$$
and thus $\delta_{i,j,k,l} = 4$ if $(i,j,k,l)\in\{  (1, 1, 1, 1), (2, 2, 0, 0), (0, 0, 2, 2) \}$. 

Moreover, $L(s, \pi_1\boxtimes \pi_1 \times \bar{ \pi}_2\boxtimes\bar{\pi}_2) 
=L(s, \pi_1\boxtimes ( \bar{ \pi}_1\otimes \bar{\chi} ) \times \pi_1\boxtimes( \bar{ \pi}_1\otimes \bar{\chi} ) )$ equals 
\begin{align*}
 L(s, (\bar{\chi} \boxplus \chi_1\bar{\chi} \boxplus (\nu/\nu^{\tau_1})\bar{\chi} \boxplus (\nu/\nu^{\tau_1})\chi_1\bar{\chi} ) \times  (\bar{\chi} \boxplus \chi_1\bar{\chi} \boxplus (\nu/\nu^{\tau_1})\bar{\chi} \boxplus  (\nu/\nu^{\tau_1})\chi_1\bar{\chi} )),
\end{align*}
which has a pole of order at most four at $s=1$ (note that $\bar{\chi}$,  $\chi_1\bar{\chi}$, $(\nu/\nu^{\tau_1})\bar{\chi}$, and $(\nu/\nu^{\tau_1})\chi_1\bar{\chi}$ are pairwise distinct). 
Similarly, we know that  $L(s, \pi_1\boxtimes \pi_1 \times  \bar{ \pi}_1\boxtimes \bar{ \pi}_2 ) 
= L(s, \pi_1 \boxtimes   \bar{ \pi}_1 \times   \pi_1 \boxtimes ( \bar{ \pi}_1\otimes \bar{\chi} ) )$ is equal to
\begin{align*}
L(s, (1 \boxplus \chi_1 \boxplus \nu/\nu^{\tau_1} \boxplus  (\nu/\nu^{\tau_1})\chi_1  ) \times  (\bar{\chi} \boxplus \chi_1\bar{\chi} \boxplus (\nu/\nu^{\tau_1})\bar{\chi} \boxplus  (\nu/\nu^{\tau_1})\chi_1\bar{\chi} )),
\end{align*}
which  has a pole of order at most four at $s=1$.

By the above discussion,  $\delta_{i,j,k,l} \le 4$ for all cases, and thus 
$$
2^2 \le
\begin{cases}
( 4 +8\cos(2\alpha) +4 +16 )     \lD(\mcS_\alpha) & \text{if $\cos(2\alpha)\ge 0$ and $\cos\alpha\ge 0$;}\\
( 4 +8\cos(2\alpha) +4 +16  -32 \cos \alpha)  \lD(\mcS_\alpha)  & \text{if  $\cos(2\alpha)\ge 0$ and $\cos\alpha \le 0$;}\\
  ( 4  +4 +16 ) \lD(\mcS_\alpha)  & \text{if  $\cos(2\alpha)\le 0$ and $\cos\alpha \ge 0$;}\\
( 4  +4 +16  -32 \cos \alpha)   \lD(\mcS_\alpha)  & \text{if  $\cos(2\alpha)\le 0$ and $\cos\alpha \le 0$.}
\end{cases}
$$
It is evident that it gives a worse lower bound for $ \lD(\mcS_\alpha)$ than most cases  given in Section \ref{XX} (except for Section \ref{1/8}, where  $\frac{1}{8}$ appears).

\subsubsection{$\pi_1$ and $\pi_2$ are not twist-equivalent} To end this section, we analyse the case of  non-twist-equivalent $\pi_1$ and $\pi_2$ that satisfy property $\mathcal{P}$.  We first recall that by \cite[Lemma 6]{Wa14},  Walji showed that both $L(s, \Ad(\pi_1)\times \Ad(\pi_1))$ and $L(s, \Ad(\pi_2)\times \Ad(\pi_2))$ have a pole of order three at $s = 1$; also, if $\pi_1$ and $\pi_2$ can be  induced from $K/F$,  $L(s, \Ad(\pi_1)\times \Ad(\pi_2))$ have a pole of order one at $s = 1$. 

In addition, if $\pi_1$ and $\pi_2$ cannot be induced from the same quadratic extension, then $\pi_1 \boxtimes \pi_2$ and $\bar{\pi}_1\boxtimes \bar{\pi}_2$ are cuspidal. So $L(s, \pi_1 \boxtimes \pi_2 \times \bar{\pi}_1\boxtimes \bar{\pi}_2)$ has a pole of order one  at $s=1$. As argued in  Section \ref{OX-not-sameK}, 
we deduce that $L(s,  \Ad(\pi_1) \times   \Ad(\pi_2) )$ is holomorphic  at $s=1$. (We note that this case was not discussed in \cite[pp. 4997-4998]{Wa14}.)

Hence,  from  \eqref{Ad-ineq}, it follows that if $\pi_1$ and $\pi_2$ are induced from the same quadratic extension of $F$,  
$$
\lD(\mathcal{S}_{\Ad})\ge \frac{3-2+3}{16} =\frac{1}{4},
$$
where $\mathcal{S}_{\Ad}$ is defined as in \eqref{SAd-def}; otherwise,
$$
\lD(\mathcal{S}_{\Ad})\ge \frac{3+3}{16} =\frac{3}{8}.
$$
In particular, for any $\alpha$, as $\mathcal{S}_{\Ad} \subseteq  \mcS_\alpha$, $\lD(\mcS_\alpha)\ge \frac{1}{4}$.

\section{Proofs of Theorems \ref{Main-dist} and \ref{main-thm-non-twist-equiv}, part II: at least one of $\pi_1$ and $\pi_2$ is non-dihedral}\label{one-non-dih}

In this section, we shall complete the proofs of Theorems  \ref{Main-dist} and \ref{main-thm-non-twist-equiv} by analysing the case that at least one of $\pi_1$ and $\pi_2$ is non-dihedral.

\subsection{Exactly one of $\pi_1$ and $\pi_2$ is dihedral} If $\pi_1$ is not dihedral but $\pi_2$ is dihedral, then using \eqref{n-d&d} and the fact that $\delta_{2, 0, 0, 2} =\delta_{0,2,2,0} $ and  arguing similarly as before, we obtain
$$
2^2 
 \le
 (2 + 2 \delta_{2, 0, 0, 2}\cos(2\alpha) +4  + 4 ) \lD(\mcS_\alpha)
\le 
\begin{cases}
(10 + 4 \cos(2\alpha)  )    \lD(\mcS_\alpha)  & \text{if $\cos(2\alpha)\ge 0$;}\\
10    \lD(\mcS_\alpha)  & \text{if $\cos(2\alpha)\le 0$,}
\end{cases}
$$
 and thus
$$
\lD(\mcS_\alpha) \ge
\begin{cases}
\frac{2}{ 5 +2   \cos (2\alpha)  }  & \text{if $\cos(2\alpha)\ge 0$;}\\
\frac{2}{5} & \text{if $\cos(2\alpha)\le 0$.}
\end{cases}
$$
(By the symmetry, the bound also holds if $\pi_1$ is  dihedral and $\pi_2$ is not dihedral.)

\subsection{Both $\pi_1$ and $\pi_2$  are  non-dihedral}

\subsubsection{Both $\pi_1$ and $\pi_2$  are  non-tetrahedral}  If both $\pi_1$ and $\pi_2$  are  non-tetrahedral, from \eqref{both-non-te},  we see that
$$
2^2 \le (2 + 2 \delta_{2, 0, 0, 2} \cos(2 \alpha) +2  + 8 )  \lD(\mcS_\alpha)
$$
and hence
$$
 \lD(\mcS_\alpha)\ge
\begin{cases}
  \frac{1}{3 + \cos (2\alpha)    }      & \text{if $\cos (2\alpha) \ge 0$;}\\
 \frac{1}{3} & \text{if  $\cos (2\alpha) \le 0$.}
\end{cases}
$$
 
\subsubsection{At least one of $\pi_1$ and $\pi_2$ is tetrahedral}  If at least one of $\pi_1$ and $\pi_2$ is tetrahedral, then applying \eqref{one-te}, we have
\begin{align*}
 \begin{split}
2^2 
&\le (2 + 2 \delta_{2, 0, 0, 2} \cos (2\alpha) +2  + 8  - 8 \kappa_1 \cos \alpha) \   \lD(\mcS_\alpha) \\
&\le
\begin{cases}
 (12 +4 \cos (2\alpha)   - 8 \kappa_1 \cos \alpha )   \lD(\mcS_\alpha) & \text{if $\cos (2\alpha) \ge 0$;}\\
 (12    - 8 \kappa_1 \cos \alpha) \lD(\mcS_\alpha) & \text{if  $\cos (2\alpha) \le 0$.}
\end{cases}
 \end{split} 
\end{align*}
Comparing  this with the previous case, we complete the proof for the second part of Theorem \ref{Main-dist}.
 
\subsubsection{$\pi_1$ and $\pi_2$ are not twist-equivalent} If  $\pi_1$ and $\pi_2$ are not twist-equivalent, then from \eqref{n-t-e}, it follows that
$$
2^2 \le (2 + 2\kappa_2 \cos (2\alpha) +2  + 4)   \lD(\mcS_\alpha) 
\le   (8 + 2\kappa_2 \cos (2\alpha) )   \lD(\mcS_\alpha) 
$$
 and thus
$$
\lD(\mcS_\alpha)\ge \frac{2}{ 4+ \kappa_2 \cos (2\alpha)},
$$  
 which completes the proof of Theorem \ref{main-thm-non-twist-equiv}.

\section{Proofs of Theorems \ref{lD_S*} and \ref{S''}} \label{proof-lD_S*}

In this section, we shall prove Theorems \ref{lD_S*} and \ref{S''}. Since the proof of Theorem  \ref{S''} follows the same line, we shall focus on the proof of Theorem  \ref{lD_S*} and then discuss how to modify it to prove  Theorem \ref{S''} at the end of this section.

Our key observation is that $\mcS_{\Ad}$ and 
$$
\mcS_* =\mcS_* (\pi_1,\pi_2) = \{\text{$v$ unramified for both $\pi_1$ and $\pi_2$} \mid |\lambda_{\pi_1}(v)|\neq  |\lambda_{\pi_2}(v)| \}
$$
are exactly the same (which follows from the identity that $|\lambda_{\pi_i}(v)|^2= \lambda_{\pi_i\times \bar{\pi}_i}(v)= \lambda_{\Ad(\pi_i)}(v) +1 $ for unramified $v$). As the   case that both $\pi_1$ and $\pi_2$ are dihedral is already proved in  Section \ref{Dih}, we shall only consider the situation that at least one of $\pi_1$ and $\pi_2$ is non-dihedral. 

We shall require the following proposition, which is a consequence of the work of Kim and  Shahidi \cite{KS02}, concerning the orders of poles of certain $L$-functions at $s=1$.

\begin{proposition}\label{key-prop}
Let $\pi$, $\pi_1$, and $\pi_2$   be   non-dihedral cuspidal automorphic representations for $\GL_2(\Bbb{A}_F )$ with unitary central characters $\omega$, $\omega_1$, and $\omega_2$, respectively.  Suppose that $\pi_1$ and $\pi_2$ are not twist-equivalent. Then we have
\begin{equation}\label{4Pi}
-\ord_{s=1} L(s, \Pi\times\Pi\times \Pi\times \Pi)
=  \begin{cases} 
 7  & \text{if $\pi$ is  tetrahedral;}\\
 4  & \text{if $\pi$ is octahedral;}\\
 3  & \text{if $\pi$ is not solvable polyhedral,}
 \end{cases}
\end{equation}
where $\Pi=\Ad(\pi)$.\footnote{A cuspidal automorphic representation $\pi$ for $\GL_2(\Bbb{A}_F )$ is called octahedral if it is non-dihedral and non-tetrahedral, and $\Sym^3\pi$ admits a non-trivial self-twist by a Hecke character; $\pi$ is called solvable polyhedral if it is either  dihedral, tetrahedral, or octahedral.}
Also, if $\pi_1$ is   not solvable polyhedral, then we have 
\begin{equation}\label{2Pii}
-\ord_{s=1} L(s, \Pi_1\times\Pi_1\times \Pi_2\times \Pi_2)
\le  \begin{cases} 
 1  & \text{if  $\pi_2$ is   tetrahedral;}\\
 2  & \text{otherwise,}
 \end{cases}
\end{equation}
where $\Pi_i= \Ad(\pi_i)$ for each $i$.
\end{proposition}

\begin{proof}
We first note that $\Pi$ and $\Pi_i$ are self-dual cuspidal representations for $\GL_3(\Bbb{A}_F )$. As a consequence of  Clebsch-Gordon decomposition (see, e.g, \cite[Lemma 3.3]{Wa14'}), it is known  that
$$
\Ad(\Pi) \simeq  \Ad(\pi)\boxplus  \Sym^4\pi\otimes \omega^{-2},
$$
where by \cite{Kim03}, $\Sym^4\pi$ is automorphic.  We note that 1 will never appear in  the isobaric sum of $\Ad(\Pi)=\Ad(\Ad(\pi))$. Indeed, by  the identity
$$
L(s,\Pi\times \bar{ \Pi} )= L(s,1  )L(s, \Ad( \Pi) )
$$
and the fact that $\Pi=\Ad(\pi)$ is cuspidal, the theory of Rankin-Selberg $L$-functions yields that
$$ 
1= -\ord_{s=1} L(s,\Pi\times  \bar{\Pi}) = 1 -\ord_{s=1} L(s, \Ad( \Pi)),
$$
and so $ L(s, \Ad( \Pi))$ is holomorphic at $s=1$.

Moreover, we recall the cuspidality criteria established by  Kim and  Shahidi \cite[Sec 3.2 and Theorem 3.3.7]{KS02}:
\begin{enumerate}
 \item[(i)] if $\pi$ is  tetrahedral,  one has 
$$ 
\Sym^4\pi \otimes \omega^{-2} \simeq \mu \boxplus  \mu^2  \boxplus \Sym^2\pi \otimes \omega^{-1} \simeq \mu \boxplus  \mu^2  \boxplus \Ad(\pi),
$$
where $\mu$ is a (cubic) Hecke character such that $\Ad(\pi) \otimes \mu \simeq \Ad(\pi)$;

 \item[(ii)] if $\pi$ is octahedral, one has
$$ 
\Sym^4\pi \otimes \omega^{-2} \simeq  \sigma \boxplus  \Ad( \pi) \otimes \eta
$$
for some dihedral representation $\sigma$ and (quadratic) Hecke character $\eta$;

 \item[(iii)] if $\pi$ is not solvable polyhedral,  then $\Sym^4\pi$ is cuspidal (and so is $\Sym^4\pi \otimes \omega^{-2}$).
\end{enumerate}

To prove \eqref{4Pi}, we use the following decomposition of $L$-functions:
$$
L(s, \Pi\times\Pi\times \Pi\times \Pi)=L(s, 1)L(s,\Ad(\Pi))^2 L(s,\Ad(\Pi)\times \Ad(\Pi)), 
$$
where $L(s,\Ad(\Pi))$ has no pole at $s=1$, and
$$
L(s,\Ad(\Pi)\times \Ad(\Pi))= L(s, (\Ad(\pi)\boxplus  \Sym^4\pi\otimes \omega^{-2}) \times  (\Ad(\pi)\boxplus  \Sym^4\pi\otimes \omega^{-2})).
$$  
Now, \eqref{4Pi} follows from the above-mentioned cuspidality criteria of  Kim and  Shahidi and the theory of Rankin-Selberg $L$-functions.

Similarly, we have
$$
L(s, \Pi_1\times\Pi_1\times \Pi_2\times \Pi_2)=L(s, 1)L(s,\Ad(\Pi_1))L(s,\Ad(\Pi_2))L(s,\Ad(\Pi_1)\times \Ad(\Pi_2)), 
$$
where each $L(s,\Ad(\Pi_i))$ has no pole at $s=1$, and
$$
L(s,\Ad(\Pi_1)\times \Ad(\Pi_2))= L(s, (\Ad(\pi_1)\boxplus  \Sym^4\pi_1\otimes \omega_1^{-2}) \times (\Ad(\pi_2)\boxplus  \Sym^4\pi_2\otimes \omega_2^{-2})) .
$$  
For  non-dihedral $\pi_1$, we have
$$
-\ord_{s=1} L(s, \Ad(\pi_1)  \times    \Sym^4\pi_2\otimes \omega_2^{-2})
\le  \begin{cases} 
 1  & \text{if $\pi_2$ is octahedral;}\\
 0  & \text{otherwise}
 \end{cases}
$$
(here we do not assume that $\pi_1$ is not solvable polyhedral).
Moreover, if $\pi_1$ is   not solvable polyhedral, then 
$$
-\ord_{s=1} L(s, \Sym^4\pi_1\otimes \omega_1^{-2}  \times    \Sym^4\pi_2\otimes \omega_2^{-2})
\le  \begin{cases} 
 0  & \text{if $\pi_2$ is  tetrahedral or  octahedral;}\\
 1  & \text{if $\pi_2$ is not solvable polyhedral.}
 \end{cases}
$$
Putting everything together, we obtain \eqref{2Pii}.
\end{proof}

\subsection{Exactly one of $\pi_1$ and $\pi_2$ is dihedral}\label{sec-5.1} 
  
 Suppose that $\pi_1$   is non-dihedral and  $\pi_2$ is  dihedral. As $\Ad(\pi_1)$ is cuspidal and $\Ad(\pi_2)$  is not cuspidal, $L(s, \Ad(\pi_1)\times \Ad(\pi_1))$ has a simple pole at $s=1$, and  $L(s, \Ad(\pi_1)\times \Ad(\pi_2))$ is holomorphic at $s=1$. Also, by \eqref{Ad-decomp-pi1-P} and \eqref{Ad-decomp-pi2-NP}, we know that $L(s, \Ad(\pi_2)\times \Ad(\pi_2))$  has a pole of order at least two at $s=1$.
  Now, as argued in Section \ref{OX-not-sameK}  (see, especially, \eqref{Ad-ineq}), if   the Ramanujan-Petersson conjecture holds for $\pi_1$, then we have
$$
\lD(\mcS_*)=\lD(\mcS_{\Ad})\ge \frac{1+2}{16} =\frac{3}{16}.
$$

As the Ramanujan-Petersson conjecture holds for solvable polyhedral representations  (see, e.g., \cite[Sec. 6]{Wa14'}), it remains to discuss  the case that  $\pi_1$   is  not solvable polyhedral. To do so, we shall apply the Cauchy-Schwarz inequality as follows. Note that as the Ramanujan-Petersson conjecture is valid for $\pi_2$ (and thus $|\lambda_{ \Ad(\pi_2)}(v) |\le 3$), applying the Cauchy-Schwarz inequality, we have
$$
|\lambda_{ \Ad(\pi_1)}(v) - \lambda_{\Ad(\pi_2)}(v) |^2 \le 2(|\lambda_{ \Ad(\pi_1)}(v) |^2+ |\lambda_{ \Ad(\pi_2)}(v) |^2)\le  2 |\lambda_{ \Ad(\pi_1)}(v) |^2 +18.
$$
Hence, as $\mcS_{\Ad}=\mcS_*$, we have
\begin{equation}\label{Ad-ineq'}
\sum_{v} \frac{|\lambda_{ \Ad(\pi_1)}(v) - \lambda_{\Ad(\pi_2)}(v) |^2 \chi_{\mcS_{\Ad}}(v)}{Nv^s} \le 2  \sum_{v} \frac{|\lambda_{ \Ad(\pi_1)}(v) |^2\chi_{\mcS_{*}}(v) }{Nv^s}+
  18  \sum_{v\in \mathcal{S}_{*}} \frac{1 }{Nv^s}.
\end{equation}
Applying the Cauchy-Schwarz inequality again and using ``positivity'' as in \cite[Eq. (2.2)]{Wa14'}, we can bound the first sum  on the right of \eqref{Ad-ineq'} as
\begin{align*}
  \sum_{v} \frac{|\lambda_{ \Ad(\pi_1)}(v) |^2\chi_{\mcS_{*}}(v) }{Nv^s}
&   \le 
   \Big( \sum_{v} \frac{|\lambda_{ \Ad(\pi_1)}(v) |^4}{Nv^s}  \Big) ^{\frac{1}{2}}\Big( \sum_{v \in \mcS_{*}} \frac{1 }{Nv^s}\Big)^{\frac{1}{2}}  \\
&   \le   
 \Big(\log(  L(s, \Pi_1\times\Pi_1\times\Pi_1\times\Pi_1 ) ) \Big)^{\frac{1}{2}}\Big( \sum_{v \in \mcS_{*}} \frac{1 }{Nv^s}\Big)^{\frac{1}{2}}    ,
\end{align*}
where $\Pi_1 = \Ad(\pi_1) $. Therefore,  arguing similarly as  in Section \ref{OX-not-sameK},  by \eqref{4Pi}, we obtain
$$
3\le 2\cdot 3^{\frac{1}{2}}\lD (\mcS_*)^\frac{1}{2} + 18   \lD(\mcS_{*}).
$$
A numerical calculation then yields
$$
\lD(\mcS_*) \ge \frac{5}{27} - \frac{\sqrt{19}}{54} \ge 0.1044... \ge \frac{1}{9.58}.
$$

 \subsection{Both $\pi_1$ and $\pi_2$ are non-dihedral}\label{sec-5.2}  
 
We start with noting that if both   $\pi_1$ and $\pi_2$ are  non-dihedral, then $\Ad(\pi_1)$ and $\Ad(\pi_2)$ are cuspidal and thus each $L(s, \Ad(\pi_i)\times \Ad(\pi_i))$ has a simple pole at $s=1$. Also, as $\pi_1$ and $\pi_2$ are non-twist-equivalent, $\Ad(\pi_1)\not\simeq\Ad(\pi_2)$, $L(s, \Ad(\pi_1)\times \Ad(\pi_2))$ is holomorphic at $s=1$. As argued previously, assuming the Ramanujan-Petersson conjecture (for both $\pi_1$ and $\pi_2$), we obtain
$$
\lD(\mcS_*)= \lD(\mcS_{\Ad}) \ge \frac{2}{16} = \frac{1}{8}.
$$

In the case that   the Ramanujan-Petersson conjecture is unknown, we shall modify  Walji's  strategy, discussed in Section \ref{strategy-Wa}, to prove:
\begin{theorem}\label{n-sol}
Let $\pi_1$  and $\pi_2$  be  a non-dihedral cuspidal automorphic representation for $\GL_2(\Bbb{A}_F )$ with unitary central characters. Suppose, further, that $\pi_1$ and $\pi_2$ are not twist-equivalent. If $\pi_1$ is not solvable polyhedral, then
$$
\lD(\mcS_*)\ge 
 \begin{cases} 
   \frac{1}{10.17} & \text{if  $\pi_2$ is   tetrahedral;}\\
   \frac{1}{10.76}  & \text{if  $\pi_2$ is octahedral;}\\ 
   \frac{1}{9.9} & \text{if  $\pi_2$ is not solvable polyhedral.}
 \end{cases}
$$  

\end{theorem} 

\begin{proof}
We first remark that when at least one of $\pi_1$ and $\pi_2$ is not solvable polyhedral, the information on at least one of $L(s, \Pi_1\times\Pi_1\times \Pi_1\times \Pi_2)$ and $L(s, \Pi_1\times\Pi_2\times \Pi_2\times \Pi_2)$ at $s=1$  seems unavailable at present. Therefore, instead of using an estimate similar to \eqref{C-S-ineq} directly, we shall require a modification.

From the inequality
$$ 
|\lambda_{ \Ad(\pi_1)}(v) - \lambda_{\Ad(\pi_2)}(v)|^2 \le |\lambda_{ \Ad(\pi_1)}(v) |^2 + 2 |\lambda_{ \Ad(\pi_1)}(v) \lambda_{ \Ad(\pi_2)}(v) | +   |\lambda_{ \Ad(\pi_2)}(v) |^2 
$$
and the Cauchy-Schwarz inequality, it follows that
\begin{align*}
 \begin{split}
&\sum_{v} \frac{|\lambda_{ \Ad(\pi_1)}(v) - \lambda_{\Ad(\pi_2)}(v) |^2 \chi_{\mcS_{\Ad}}(v)}{Nv^s} \\
&\le   \Big(  \Big( \sum_{v} \frac{|\lambda_{ \Pi_1}(v) |^4 }{Nv^s} \Big)^{\frac{1}{2}}+ 2 \Big( \sum_{v} \frac{|\lambda_{ \Pi_1 \times \Pi_2}(v) |^2 }{Nv^s} \Big)^{\frac{1}{2}} + \Big( \sum_{v} \frac{|\lambda_{ \Pi_2}(v) |^4 }{Nv^s} \Big)^{\frac{1}{2}} \Big) \Big( \sum_{v \in \mcS_{\Ad}} \frac{1 }{Nv^s}\Big)^{\frac{1}{2}},
 \end{split} 
\end{align*} 
where, as before, $\Pi_i=\Ad(\pi_i)$.  From Proposition \ref{key-prop}, arguing similarly as before, we derive
$$
\lD(\mcS_*)= \lD(\mcS_{\Ad})\ge 
 \begin{cases} 
  \frac{4}{(\sqrt{3}+2 +\sqrt{7})^2}\ge \frac{1}{10.17} & \text{if  $\pi_2$ is   tetrahedral;}\\
  \frac{4}{(\sqrt{3}+ 2\sqrt{2}  +\sqrt{4})^2} \ge\frac{1}{10.76}  & \text{if  $\pi_2$ is octahedral;}\\ 
   \frac{4}{(\sqrt{3}+ 2\sqrt{2}+\sqrt{3})^2} \ge \frac{1}{9.9} & \text{if  $\pi_2$ is not solvable polyhedral.}
 \end{cases}
$$  
Herein, we conclude the proof.
\end{proof}

\subsection{Proof of Theorem  \ref{S''}} 

To prove Theorem \ref{S''}, we consider the set
$$
\mcS_{\Ad}^{-}  =\mcS_{\Ad}^{-} (\pi_1,\pi_2) = \{\text{$v$ unramified for both $\pi_1$ and $\pi_2$} \mid \lambda_{\Ad(\pi_1)}(v) \neq  -\lambda_{\Ad( \pi_2)}(v) \}.
$$ 
and the sum
$$
\sum_{v} \frac{|\lambda_{ \Ad(\pi_1)}(v) + \lambda_{\Ad(\pi_2)}(v) |^2 \chi_{\mcS_{\Ad}^-}(v)}{Nv^s},
$$
where $\chi_{\mcS_{\Ad}^-}$ is the indicator function of $\mcS_{\Ad}^-$. Note that the Ramanujan-Petersson conjecture  gives 
$$
 |\lambda_{ \Ad(\pi_1)}(v) + \lambda_{\Ad(\pi_2)}(v) |^2 \le (3+3)^2\le 36.
$$
Processing a similar argument as in the previous sections (including Section \ref{Dih}) then results in
$$
\lD(\mcS_{\Ad}^-)\ge
 \begin{cases} 
   \frac{1}{18}  & \text{if  $\pi_1$ and $\pi_2$ are simultaneously dihedral or non-dihedral;}\\
   \frac{1}{12}  & \text{if exactly  one of $\pi_1$ and $\pi_2$ is dihedral.}
 \end{cases}
$$  
Finally, observing that
$$
\lambda_{ \Ad(\pi_1)}(v) + \lambda_{\Ad(\pi_2)}(v) = |\lambda_{\pi_1}(v)|^2 + |\lambda_{\pi_2}(v)|^2 -2,
$$
we completes the proof of Theorem \ref{S''}.

\begin{remark}
Although it seems that one would obtain a better bound for $\lD(\mcS^*)$ in the simultaneously non-dihedral case than the simultaneously dihedral case by applying the argument used in 
Section \ref{sec-5.2}, it is not always the case. Indeed, if both $\pi_1$ and $\pi_2$ are tetrahedral, $L(s,\Pi_1\times\Pi_1\times\Pi_2\times\Pi_2)$ may admit a pole of order three at $s=1$, where $\Pi_i=\Ad(\pi_i)$. This results in $\lD(\mcS^*)=\lD(\mcS_{\Ad}^-)\ge\frac{4}{(\sqrt{7}+2\cdot3+\sqrt{7})^2}\approx \frac{1}{31.88}$. Nonetheless, if $\pi_1$ is not solvable polyhedral (regardless of that $\pi_2$ is non-dihedral or not), arguing as in Sections \ref{sec-5.1} and \ref{sec-5.2}, one can obtain $\lD(\mcS^*)=\lD(\mcS_{\Ad}^-)\ge \frac{1}{10.76}$.  
\end{remark}

\begin{remark}
Let $n\ge 3$, and let $\pi_1$ and $\pi_2$ be distinct cuspidal automorphic representations for $\GL_n(\Bbb{A}_F )$, satisfying the Ramanujan-Petersson conjecture, such that $\Ad(\pi_1)$ and $\Ad(\pi_2)$ are cuspidal. Applying the Cauchy-Schwarz inequality twice, we have
\begin{align}\label{lastCS}
  \begin{split}
\sum_{v} \frac{|\lambda_{ \pi_1}(v) - e^{i\alpha} \lambda_{\pi_2}(v) |^2 \chi_{\mcS_{\alpha}}(v)}{Nv^s} 
&\le 2\sum_{i=1}^2  \sum_{v} \frac{|\lambda_{ \pi_i}(v) |^2\chi_{\mcS_{\alpha}}(v) }{Nv^s}\\
&\le 2\sum_{i=1}^2 \Big( \sum_{v} \frac{|\lambda_{ \pi_i}(v) |^4}{Nv^s}  \Big) ^{\frac{1}{2}}\Big( \sum_{v \in \mcS_{\alpha}} \frac{1 }{Nv^s}\Big)^{\frac{1}{2}},
 \end{split}
\end{align}
where $\mathcal{S}_\alpha= \mathcal{S}_\alpha(\pi_1,\pi_2) =\{\text{$v$ unramified for both $\pi_1$ and $\pi_2$} \mid \lambda_{\pi_1}(v)\neq e^{i\alpha} \lambda_{\pi_2}(v) \}.$
Since
$$
L(s, \pi_i\times \bar{\pi}_i\times \pi_i\times \bar{\pi}_i)=L(s, 1)L(s,\Ad(\pi_i))^2 L(s,\Ad(\pi_i)\times \Ad(\pi_i)), 
$$
$L(s, \pi_i\times \bar{\pi}_i\times \pi_i\times \bar{\pi}_i)$ has a pole of order two at $s=1$.  Hence, we have
$$
\lD(\mcS_\alpha)\ge \Big(\frac{2}{4\sqrt{2}}\Big)^2 = \frac{1}{8}.
$$

Furthermore, without the assumption that $\Ad(\pi_1)$ and $\Ad(\pi_2)$ are automorphic, as each $\Ad(\pi_i)$ 
satisfies the Ramanujan-Petersson conjecture, we know that $|\lambda_{ \pi_i}(v) |^2 \le n^2$. Thus, applying the first inequality in \eqref{lastCS}, we obtain
$$
\lD(\mcS_\alpha)\ge \frac{2}{4n^2}= \frac{1}{2n^2}.
$$

\end{remark}

\section{Proofs of Theorems  \ref{main-uncond}, \ref{main-2}, and \ref{main-GRH}}\label{main-proof}

In this section, we will prove Theorems  \ref{main-uncond}, \ref{main-2}, and \ref{main-GRH}. For the sake of convenience, we  shall assume that any prime $v$ in the consideration is unramified for both $\pi_1$ and $\pi_2$. (Note the argument presented in Section \ref{gen} does not make the assumption that both $\pi_1$ and $\pi_2$ are non-dihedral.)

\subsection{Generality}\label{gen}
For each $i$, we write
$$
\lambda_{\pi_i}(v)=2\cos  \theta_{i,v}
$$
for some $\theta_{i,v}\in [0,\pi]$. It is clear that
$$
\lambda_{\pi_1}(v)=\lambda_{\pi_2}(v)\text{ if and only if }   \theta_{1,v}= \theta_{2,v}
$$
and that
$$
\lambda_{\pi_1}(v)=-\lambda_{\pi_2}(v)\text{ if and only if }   \theta_{1,v}= \pi-\theta_{2,v}.
$$
Thus, $\#\{v \mid  Nv\le x,\enspace \lambda_{\pi_1}(v)^2= \lambda_{\pi_2}(v)^2 \}$ is less than or equal to
\begin{equation}\label{start-decomp}
\#\{v \mid  Nv\le x,\enspace  \theta_{1,v}=\theta_{2,v}  \} +\#\{v \mid  Nv\le x,\enspace \theta_{1,v}=\pi-\theta_{2,v}  \}. 
\end{equation} 

In  the notation of Section \ref{jSTd-Sp},  for any   $M\ge 1$ and $\delta\in (0,\frac{1}{\pi(M+1)}]$, we first observe that
\begin{equation}\label{prime-counting}
\#\{v \mid  Nv\le x,\enspace \theta_{1,v}=\theta_{2,v}  \} \le  \sum_{Nv\le x} \Big( \mathcal{I}_{\delta}(\theta_{1,v}-\theta_{2,v}) +\mathcal{I}_{\delta}(\theta_{1,v}+\theta_{2,v}) \Big) .
\end{equation}
From Proposition \ref{Sel-poly} and the identity
$$
\cos(n(\theta_{1,v}-\theta_{2,v} ))+\cos(n(\theta_{1,v} +\theta_{2,v})) =2 \cos(n\theta_{1,v})\cos(n\theta_{2,v}),
$$
it follows that the sum on the right of \eqref{prime-counting} is less than or equal to 
\begin{equation}\label{1st-decomp}
\Big(2\delta+ \frac{2}{M+1}\Big)\pi_F(x)+  4\sum_{n=1}^M \Re(\hat{S}^+_{J,M}(n) )  \sum_{Nv\le x}\cos(n\theta_{1,v})\cos(n\theta_{2,v}),
\end{equation}
where $\pi_F(x)$ denotes the number of primes $v$ of $F$ such that $Nv\le x$. Since  $\delta\in (0,\frac{1}{\pi(M+1)}]$,  $|\Re(\hat{S}^+_{J,M}(n) )| \le \delta+ \frac{1}{M+1}\le \frac{2}{M}$, and thus
\begin{equation}\label{n=1,2}
4 \sum_{n=1}^2  \Re(\hat{S}^+_{J,M}(n) ) \sum_{Nv\le x} \cos(n\theta_{1,v})\cos(n\theta_{2,v})\ll  \frac{1}{M} \pi_F(x).
\end{equation}
Also, recalling that for $n\ge 2$,
$$
2 \cos (n\theta) =\frac{\sin ((n+1)\theta)}{\sin\theta} -\frac{\sin((n-1)\theta)}{\sin\theta}= U_n(\cos\theta) -U_{n-2}(\cos\theta),
$$
we see that the remaining terms\footnote{We note that although  \cite{MP17}  includes the case $n=2$ in their consideration, one has to treat it separately since \cite[Proposition 2.1]{MP17} (cf. Propositions \ref{uncond-ST}) does not cover the situation that  $U_{2-2}(\cos(\theta_{1,v}))U_{2-2}(\cos(\theta_{2,v}))\equiv 1.$} in the double sum of \eqref{1st-decomp} become
\begin{equation}\label{n>3}
\ll \frac{1}{M} \sum_{n=3}^M  \Big|  \sum_{Nv\le x} \Big( U_n(\cos(\theta_{1,v})) -U_{n-2}(\cos(\theta_{1,v}))\Big) 
\Big( U_n(\cos(\theta_{2,v})) -U_{n-2}(\cos(\theta_{2,v}))\Big)  \Big|.
\end{equation}
Finally, gathering \eqref{prime-counting}, \eqref{1st-decomp}, \eqref{n=1,2}, and \eqref{n>3}, we see that $\#\{v \mid  Nv\le x,\enspace \theta_{1,v}=\theta_{2,v}  \}$
\begin{align}\label{"+"}
 \ll \frac{\pi_F(x)}{ M}  + \frac{1}{M} \sum_{n=3}^M   \Big|  \sum_{Nv\le x} \prod_{i=1}^2\Big( U_n(\cos(\theta_{i,v})) -U_{n-2}(\cos(\theta_{i,v}))\Big) 
  \Big|.
\end{align}

To bound  $\#\{v \mid  Nv\le x,\enspace  \theta_{1,v}=\pi-\theta_{2,v}  \} $, we use the estimate
$$
\#\{v \mid  Nv\le x,\enspace \theta_{1,v}=\pi-\theta_{2,v}  \} \le \sum_{Nv\le x} \Big( \mathcal{I}_{\delta}(\theta_{1,v}-(\pi -\theta_{2,v} ))+\mathcal{I}_{\delta}(\theta_{1,v} + (\pi -\theta_{2,v} )) \Big).
$$
As $\cos (n (\theta_{1,v}-(\pi -\theta_{2,v} ) ) )+ \cos (n( \theta_{1,v} + (\pi -\theta_{2,v} )))$ equals to
\begin{align*}
 2 \cos(n \theta_{1,v} )\cos  (n (\pi -\theta_{2,v} ))=  2 (-1)^n\cos(n \theta_{1,v} )\cos  (n \theta_{2,v} ),
\end{align*}
by Proposition \ref{Sel-poly}, we see that $\#\{v \mid  Nv\le x,\enspace \theta_{1,v}=\pi-\theta_{2,v}  \}$ is less than or equal to
$$
\Big(2\delta+ \frac{2}{M+1}\Big)\pi_F(x)+  4\sum_{n=1}^M \Re(\hat{S}^+_{J,M}(n) ) (-1)^n \sum_{Nv\le x}\cos(n\theta_{1,v})\cos(n\theta_{2,v}). 
$$
By an analogous argument as above, this becomes
\begin{align}\label{"-"}
 \ll \frac{\pi_F(x)}{ M}  + \frac{1}{M} \sum_{n=3}^M   \Big|  \sum_{Nv\le x} \prod_{i=1}^2\Big( U_n(\cos(\theta_{i,v})) -U_{n-2}(\cos(\theta_{i,v}))\Big) 
  \Big|.
\end{align}

\subsection{Proof of Theorem  \ref{main-uncond}}\label{uncond-proof}
Now we are in  a position to prove Theorem  \ref{main-uncond}. Assume that both $\pi_1$ and $\pi_2$ are non-dihedral and that
\begin{equation}\label{assump}
\limsup_{x\rightarrow\infty}\frac{\#\{v \mid  Nv\le x,\enspace \lambda_{\pi_1}(v)^2 = \lambda_{\pi_2}(v)^2 \}}{\pi_F(x)}>0.
\end{equation}
Suppose, on the contrary, that $\pi_1$ and  $\pi_2$ are not twist-equivalent. From  Proposition \ref{uncond-ST} and estimates \eqref{"+"} and \eqref{"-"}, it follows that   the limit on the left of \eqref{assump} is less than or equal to
\begin{align*}
 \limsup_{x\rightarrow\infty}\frac{\#\{v \mid  Nv\le x,\enspace  \theta_{1,v}=\theta_{2,v}  \}}{\pi_F(x)}+\limsup_{x\rightarrow\infty}\frac{\#\{v \mid  Nv\le x,\enspace  \theta_{1,v}=\pi-\theta_{2,v}  \}}{\pi_F(x)} \ll \frac{1}{M} 
\end{align*}
for any $M\ge 1$. Making $M\rightarrow \infty$ yields that
$$
\limsup_{x\rightarrow\infty}\frac{\#\{v \mid  Nv\le x,\enspace \lambda_{\pi_1}(v)^2 = \lambda_{\pi_2}(v)^2 \}}{\pi_F(x)}=0
$$
and thus $0<0$, a contradiction.

\subsection{Completing the proof of Theorems   \ref{main-2}  and \ref{main-GRH}} 
Let $F=\Bbb{Q}$. For each $i$, let $\pi_i$ be the cuspidal automorphic representation  corresponding to a  non-CM  newform  in $S_{k_i}^{\mathrm{new}}(\Gamma_0(q_i))$  with trivial nebentypus.  Assume that $\pi_1$ and $\pi_2$ are not twist-equivalent. Applying  Proposition \ref{Th2} (and the recent work of Newton and Thorne \cite{NT20} on the automorphy of symmetric powers of $\pi_i$), we have
\begin{align}\label{2nd-term}
 \begin{split}
\frac{1}{M} \sum_{n=3}^M  & \Big|  \sum_{p\le x} \prod_{i=1}^2\Big( U_n(\cos(\theta_{i,p})) -U_{n-2}(\cos(\theta_{i,p}))\Big)\Big| \ll \pi(x)\exp\Big(\frac{-c_2 \log x}{(k_1q_1k_2q_2  M)^{ c_3 M^2}}\Big)\\ 
&+ M^4 \pi(x)\Big( \exp\Big(\frac{- \log x}{c_4 M^2} \Big)
+ \exp\Big(\frac{-c_5\log x}{ M^2\log (k_1q_1k_2q_2 M) } \Big)+\exp\Big(\frac{ -c_5 \sqrt{\log x}  }{ M} \Big)    \Big).
  \end{split}
\end{align}
Thus, using \eqref{start-decomp}, \eqref{"+"}, \eqref{"-"}, and \eqref{2nd-term}  and choosing 
$$
 M = \ceil{ c_{6}  \sqrt{\log \log x}/ \log(k_1q_1k_2q_2 \log \log x) } \ge 3, 
$$
for some sufficiently small $c_6>0$, we arrive at
$$
\#\{p\le x \mid  \lambda_{\pi_1}(p)^2= \lambda_{\pi_2}(p)^2 \} \ll \pi(x)\frac{ \log(k_1q_1k_2q_2 \log \log x) }{\sqrt{\log\log x}}.
$$

Moreover, assuming the generalised Riemann hypothesis, by Proposition \ref{PJW} and  estimates \eqref{start-decomp}, \eqref{"+"}, and \eqref{"-"}, we have 
$$
\#\{p\le x \mid  \lambda_{\pi_1}(p)^2= \lambda_{\pi_2}(p)^2 \}
  \ll \frac{1}{M}\frac{x}{\log x} + M^2 x^{1/2}  \log( (k_1q_1k_2q_2 M) x).
$$
Choosing 
$$
M = \ceil{ x^{1/6}/(\log x)^{1/3} (\log( (k_1q_1k_2q_2 ) x))^{1/3} }, 
$$
we derive
$$
\#\{p\le x \mid  \lambda_{\pi_1}(p)^2= \lambda_{\pi_2}(p)^2 \}  \ll \frac{x^{5/6}(\log(  k_1q_1k_2q_2   x))^{1/3}}{(\log x)^{2/3}}.
$$

\section{Remarks on non-CM condition in Theorems \ref{main-uncond}, \ref{main-TE-2}, \ref{main-2}, and \ref{main-GRH}}\label{con-rmk}

We shall note that in Theorems \ref{main-uncond}, \ref{main-TE-2}, \ref{main-2}, and \ref{main-GRH},  we assume   $\pi_1$ and $\pi_2$  correspond to non-CM newforms $f_1$ and $f_2$, respectively, just for the sake of simplicity of discussion. As done in  \cite{MP17,R98,Rama00}, it is possible to drop the assumption that $f_2$ is without CM. (We note that as remarked in \cite{MP17} if both $f_1$ and $f_2$ are with CM, then the theorems  are not always true.) 

To extend Theorem \ref{main-uncond}, we require the following estimate.

\begin{proposition}\label{uncond-ST'}
Let $F$ be a totally real number field.  For each $i$, let $\pi_i$ be a cuspidal automorphic representation  corresponding to a  Hilbert newform $f_i$  of weights $k_{i,j}\ge 2$ (at all infinite primes $v_j$ of $F$) and with trivial nebentypus. Let $m_i\ge 1$. Assume that $f_1$ is a non-CM newform such that $\Sym^{m_1}\pi_1$ is automorphic, and suppose that $f_2$ is with CM. Then one has
\begin{align*}
\sum_{Nv\le x} U_{m_1}(\cos\theta_{1,v})U_{m_2}(\cos\theta_{2,v})  =  o(\pi_F(x) ).
\end{align*}
\end{proposition}

We note that this proposition was stated in \cite[Proposition 2.1]{MP17} (for $F=\Bbb{Q}$) without assuming the automorphy of $\Sym^{m_1}\pi_1$. However, as may be noticed, the argument in \cite[Sec. 4]{MP17} only works for the case that both $f_1$ and $f_2$ are without CM. Indeed, as argued in \cite[Sec. 4]{MP17} (see also \cite[pp. 716-719, especially, Theorems 2.4 and 2.5]{Harris14} and \cite[Proof of Theorem 1.1]{PJW19}),  to use the ``Brauer-Taylor induction'' (together with the automorphy theorem from \cite{BLGHT11} and the theory of Rankin-Selberg $L$-functions), one would require the \emph{cuspidality} of  both $\Sym^{m_1}\pi_1$ and $\Sym^{m_2}\pi_2$ after a suitable base change, which may not hold if $f_2$ is with CM (see the proof below).
 
We also note that the proof of Hecke's theorem on the distribution of the Frobenius angles of CM modular forms does not rely on the symmetric power $L$-functions but the equidistribution of the values of Hecke characters (see, e.g., \cite[Theorem 3.1.1 and Sec. 3.3]{AIW}). 

Since we did not find a reference with the precise proof for the proposition (although it seems to be known by experts, at least, implicitly), we include a proof in this section for the sake of completeness.

\begin{proof}[Proof of Proposition \ref{uncond-ST'}]
As argued in \cite{MP17}, by the Wiener-Ikehara Tauberian theorem (see, e.g., \cite{RM05}), to prove the proposition, it suffices to show that the $L$-function
$$
L(s, \Sym^{m_1} \pi_1 \times \Sym^{m_2}\pi_2 )
$$
extends to a non-vanishing holomorphic function on $\Re(s)\ge 1$.  As remarked in \cite[pp. 243  and 251]{Rama07}, if $\pi_2$ corresponds to a CM Hilbert newform $f_2$ (and so $\pi_2$ is dihedral), there exists an imaginary quadratic extension $K$ of $F$ such that
$$
\Sym^{m_2} \pi_2 \simeq\boxplus_j  \Pi_j,
$$
where $\Pi_j$ is either  an id\`ele class character of $F$ or a two-dimensional cuspidal automorphic representation induced from a (non-trivial) character of $K$. 
Thus, we have the factorisation
$$
L(s, \Sym^{m_1} \pi_1 \times \Sym^{m_2}\pi_2 )=\prod_j L(s, \Sym^{m_1} \pi_1 \times \Pi_j).
$$

By the theory of Rankin-Selberg $L$-functions, each $L(s, \Sym^{m_1} \pi_1 \times \Pi_j)$ extends holomorphically to $\Re(s)\ge 1$ except for a possible simple pole at $s=1-it$ that exists only if $\Sym^{m_1} \pi_1 \simeq \bar{ \Pi}_j  \otimes |\cdot|^{it}$. By a dimension consideration, if $L(s, \Sym^{m_1} \pi_1 \times \Pi_j)$ admits a pole at $s=1-it$, then $m_1 =1$ and $\dim \Pi_j=2$ and thus $\pi_1 \simeq \bar{ \Pi}_j \otimes |\cdot|^{it}$. However, it is impossible as $f_1$  is without CM (so $\pi_1$ is non-dihedral), but $\Pi_j$ is induced from a character (so it is dihedral). Finally, it follows from the work of Shahidi \cite{Sh81}, each $L(s, \Sym^{m_1} \pi_1 \times \Pi_j)$ is non-vanishing on $\Re(s)\ge 1$, which concludes the proof. 
\end{proof} 
 

Now, using Proposition \ref{uncond-ST'} in the place of Proposition \ref{uncond-ST} in the proof of Theorem  \ref{main-uncond} given in Section \ref{uncond-proof}, we have the following:

 \begin{theorem}\label{main-3}
Let $F$ be a totally real number field.  For each $i$, let $\pi_i$ be a cuspidal automorphic representation  corresponding to a  Hilbert newform $f_i$  of weights $k_{i,j}\ge 2$ and with trivial nebentypus. Assume that $f_1$ is without CM  and that all the symmetric powers $\Sym^{m_1}\pi_1$ are automorphic. If
$$
\limsup_{x\rightarrow\infty}\frac{\#\{v \mid  Nv\le x,\enspace \lambda_{\pi_1}(v)^2= \lambda_{\pi_2}(v)^2 \}}{\pi(x)}>0,
$$
then $\pi_1$ and $\pi_2$ are twist-equivalent. 
\end{theorem}

In closing this section, we remark that it is possible to prove an effective version of Proposition \ref{uncond-ST'} by adapting the methods used in \cite{Th20,PJW19} and thus obtain a version of Theorems \ref{main-2} and \ref{main-GRH}, without assuming that $\pi_2$ is without CM, by using the argument developed in Section \ref{main-proof}.

\section*{Acknowledgments}

The author would like to thank Professor Amir Akbary and Wen-Ching Winnie Li for making several helpful suggestions. He is also grateful to the referee for careful reading and valuable comments and suggestions.


\begin{thebibliography}{99}

\bibitem{AIW} S. Arias-de-Reyna, I. Inam, and G. Wiese, On conjectures of Sato-Tate and Bruinier-Kohnien, Ramanujan J. 36 (2015), 455-481.


\bibitem{BLGG11} T. Barnet-Lamb, T. Gee, and D. Geraghty, The Sato-Tate conjecture for Hilbert modular forms, J. Amer. Math. Soc. 24 (2) (2011) 411-469.


\bibitem{BGGT11} T. Barnet-Lamb, T. Gee, D. Geraghty, and R. Taylor, Potential automorphy and change of weight, Ann. Math. 179 (2014), 501-609.


\bibitem{BLGHT11} T. Barnet-Lamb, D. Geraghty, M. Harris, and R. Taylor, A family of Calabi-Yau varieties and potential automorphy II, Publ. Res. Inst. Math. Sc, 47 (2011), 29-98.



\bibitem{Bl06} D. Blasius, Hilbert modular forms and the Ramanujan conjecture,  Noncommutative Geometry and Number Theory, pp. 35-56,  Aspects Math., vol. E37, Friedr. Vieweg, Wiesbaden, 2006.

\bibitem{BlBr11} V. Blomer and F. Brumley, On the Ramanujan conjecture over number fields, Ann.
of Math. (2) 174 (2011),  581-605.


\bibitem{GJ} S. Gelbart and H. Jacquet, A relation between automorphic representations of $\GL(2)$ and $\GL(3)$, Ann. Sci. \'Ec. Norm. Sup\'er. 11 (1978), 471-542.

\bibitem{Harris14} M. Harris, Galois representations, automorphic forms, and the Sato–Tate conjecture, Indian J. Pure Appl. Math. 45 (5) (2014), 707-746.

\bibitem{JP-SS} H. Jacquet, I. Piatetski-Shapiro, J. Shalika, Rankin-Selberg convolutions, Amer. J. Math. 105 (1983), 367-464.

\bibitem{JS76} H. Jacquet and J. A. Shalika, A non-vanishing theorem for zeta functions of GLn,
Invent. Math. 38 (1976/77), 1-16.


\bibitem{JS81} H. Jacquet and J. A. Shalika, On Euler products and the classification of automorphic forms. II, Amer. J. Math. 103 (1981),  777-815.

\bibitem{Kim03} H. H. Kim, Functoriality for the exterior square of $\GL_4$ and the symmetric
fourth of GL2, J. Amer. Math. Soc. 16(1) (2003), 139-183. With appendix 1 by Dinakar Ramakrishnan and appendix 2 by Kim and Peter Sarnak.


\bibitem{KS00} H. H. Kim and F. Shahidi, Functorial products for $\GL_2 \times \GL_3$ and functorial symmetric cube for $\GL_2$, C. R. Acad. Sci. Paris S\'er. I Math. 331 (2000), 599-604.

\bibitem{KS02} H. H. Kim and F. Shahidi, Cuspidality of symmetric powers with applications, Duke Math. J. 112 (2002), 177-197.

\bibitem{Mon94} H. L. Montgomery, Ten Lectures on the Interface between Analytic Number Theory and Harmonic  Analysis, volume 84 of CBMS Regional Conference Series in Mathematics, published for the Conference Board of the Mathematical Sciences, Washington, DC, Amer. Math. Soc., Providence, RI, 1994.

\bibitem{RM05} M. Ram Murty, Problems in analytic number theory, 2nd edition, Springer, New York, 2005.

\bibitem{MP17} R. M. Murty and  S. Pujahari, Distinguishing Hecke eigenforms,  	Proc. Amer. Math. Soc.  145 (2017), 1899-1904.

\bibitem{MR96} M. R. Murty and C. S. Rajan, Stronger multiplicity one theorems for forms of general type on $\GL_2$, Analytic number theory, Vol. 2 (Allerton Park, IL, 1995),  pp. 669-683, Progr. Math., vol. 139, Birkh\"auser Boston, Boston, MA, 1996.


\bibitem{NT20} J. Newton and J. A. Thorne, Symmetric power functoriality for holomorphic modular forms, II, arXiv:2009.07180.


\bibitem{R98} C. S. Rajan, On strong multiplicity one for $l$-adic representations, Internat. Math. Res. Notices 3 (1998), 161-172.



\bibitem{Rama94} D. Ramakrishnan, Appendix: a refinement of the strong multiplicity one theorem for $\GL(2)$, Invent. Math. 116 (1994), 645-649.

\bibitem{Rama00} D. Ramakrishnan, Appendix: recovering modular forms from squares, Invent. Math. 139 (2000), 29-39. 

\bibitem{Rama-ann-00} D. Ramakrishnan, Modularity of the Rankin-Selberg $L$-series, and multiplicity one for $\SL(2)$, Ann. of Math. (2) 152 (2000), 45-111.





\bibitem{Rama07} D. Ramakrishnan, Remarks on the symmetric powers of cusp forms on $\GL (2)$, Automorphic Forms and L-Functions I. Global Aspects (In honor of Steve Gelbart), pp. 237-256, Contemp. Math., 488, Amer. Math. Soc., Providence, RI, 2009.




\bibitem{Se77} J.-P. Serre, Modular forms of weight one and Galois representations, Algebraic Number Fields: $L$-Functions and Galois Properties (Proc. Sympos., Univ. Durham, Durham, 1975), pp. 193-268, Academic Press, London, 1977.




\bibitem{Sh81} F. Shahidi, On certain $L$-functions, Amer. J. Math. 103 (1981),  297-355.



\bibitem{Th20} J. Thorner, Effective forms of the Sato-Tate conjecture, arXiv:2002.10450.


\bibitem{Wa14} N. Walji, Further refinement of strong multiplicity one for $\GL(2)$, Trans. Amer. Math. Soc. 366 (2014), 4987-5007.

\bibitem{Wa14'} N. Walji, On the occurrence of Hecke eigenvalues and a lacunarity question of Serre,  Math. Res. Lett.  21 (2014), 1465-1482.

\bibitem{Wa20} N. Walji, A conjectural refinement of strong multiplicity one for $\GL(n)$, arXiv:1812.11232.


\bibitem{PJW19} P.-J. Wong, On the Chebotarev-Sato-Tate phenomenon, J. Number Theory 196 (2019), 272-290.


\end{thebibliography}
\end{document}